\documentclass[12pt,reqno]{amsart}
\usepackage{amssymb,amsmath,amsthm,newlfont}
\usepackage[dvips,final]{graphicx}
\usepackage[T1]{fontenc}
\usepackage{a4wide}

\numberwithin{equation}{section}
\newcommand{\T}{\mathbb{T}}
\newcommand{\K}{\mathbb{K}}
\newcommand{\N}{\mathbb{N}}

\newcommand{\E}{\mathbb{E}}
\newcommand{\F}{\mathbb{F}}
\newcommand{\C}{\mathbb{C}}

\newcommand{\D}{\mathbb{D}}
\newcommand{\Z}{\mathbb{Z}}

\newcommand{\cJ}{{\mathcal{J}}}
\newcommand{\cI}{{\mathcal{I}}}

\newcommand{\cA}{{\mathcal{A}}}
\newcommand{\cH}{{\mathcal{H}}}
\newcommand{\cK}{{\mathcal{K}}}

\newcommand{\cL}{{Lip_\omega}}

\newtheorem{thm}{\bf Theorem}[]
\newtheorem{coro}{\bf Corollary}

\newtheorem{prop}[thm]{\bf Proposition}

\newtheorem{lem}{\bf Lemma}

\begin{document}
\title{Weighted Big Lipschitz algebras of analytic functions and closed ideals.}
\date{\today}
\author[B. Bouya and M. Zarrabi]{Brahim Bouya and Mohamed Zarrabi}

\address{B. BOUYA \\
Madretschstr.52, CH-2503, Biel.}
\email{brahimbouya@gmail.com}

\address{M. ZARRABI \\
Institut De Math\'ematiques de Bordeaux (IMB), Universit\'e Bordeaux,
351 cours de la lib\'eration, 33405 Talence, France.}
\email{Mohamed.Zarrabi@math.u-bordeaux1.fr}

\thanks{This paper was partially supported by ANR FRAB: ANR-09-BLAN-0058-02}
\thanks{The first author wishes to thank the ICTP of Trieste, where a part of this paper was achieved during his stay as a post-doctoral researcher in mathematics section.}

\maketitle

\begin{abstract}
{We give the smallest closed ideal with given hull and inner factor for some weighted big Lipschitz algebras of analytic functions.}
\end{abstract}

\section{\bf Introduction.}

Let $\cH^{\infty}$ be the space of all bounded analytic functions on the open unit disk $\D$ of the complex plane.
Various authors have characterized completely the structure of the closed ideals in some separable Banach algebras of analytic functions in $\D$, we refer the reader for instance to \cite{Bou1,Kor,Mat,Sha,Shi} and some references therein.  They have proved that this structure is standard in the sense of that given by Beurling and Rudin in the usual disk algebra $\cA(\D)$, of functions $f\in\cH^{\infty}$ that are continuous up to the boundary $\T,$ see  \cite[Page. 85]{Hof}  and \cite{Rud}. However, the structure of closed ideals in nonseparable Banach algebras of analytic functions seems much more difficult to characterize,  we can see for instance \cite{BZ, GIR, Hed, Hof2, Ped, She}.
In this paper we are interested in the description of the closed ideals in the following weighted big Lipschitz algebra
$$\cL(\D):=\Big\{f\in\cH^\infty\text{ :
}\sup_{z,w\in\D\atop z\neq w}\frac{|f(z)-f(w)|}{\omega(|z-w|)}<+\infty\Big\}, $$ where $\omega\neq0$ is a {\it modulus of
continuity}, i.e., a nonnegative nondecreasing continuous function
on $[0,2]$ with $\omega(0)=0$ and  $t\mapsto \omega(t)/t$ is  nonincreasing.
Our main result is stated in Theorem \ref{main1} below.

It is well-known that $\cL$ is a nonseparable commutative Banach algebra when equipped with the norm
$$\|f\|_{\omega,\D}:=\|f\|_\infty + \sup_{z,w\in\D\atop z\neq w}\frac{|f(z)-f(w)|}{\omega(|z-w|)},$$
where  $\|f\|_\infty:=\sup\limits_{z\in\D}|f(z)|$ is the supremum norm.
Similarly the weighted big Lipschitz algebra $\cL(\T)$ is defined by
$$\cL(\T):=\Big\{f\in\cA(\D)\text{ :
}\sup_{
z,w\in\T\atop z\neq w}\frac{|f(z)-f(w)|}{\omega(|z-w|)}<+\infty\Big\}.$$
Tamrazov \cite{Tam} has shown that the
algebras $\cL(\D)$ and $\cL:=\cL(\T)$ coincide for any
modulus of continuity $\omega,$ see also \cite[Appendix A]{Bou1}. Furthermore, the norms $\|f\|_{\omega,\D}$ and $\|f\|_\omega:=\|f\|_{\omega,\T}$ are equivalent, where
$$\|f\|_{\omega,\T}:=\|f\|_\infty + \sup_{z,w\in\T\atop z\neq w}\frac{|f(z)-f(w)|}{\omega(|z-w|)}.$$
In what follows, we denote by  $\E$ a closed subset of $\T$ of Lebesgue measure zero and by $U\in\cH^{\infty}$  an inner function   such that $\sigma(U)\cap\T\subseteq \E,$ where
$$
\sigma(U)  :=\{\lambda\in\overline{\D}\ :\ \liminf_{z\rightarrow\lambda\atop z\in\D }|U(z)|=0\},\\
$$
is called the spectrum of $U,$   see for instance \cite[Pages. 62-63]{Nik}.
It is known that $\sigma(U)=\overline {\Z_U} \cup {\text{supp}} (\mu_{_U}),$
where  $\Z_U$ is the zero set in $\D$ of $U$ and ${\text{supp}} (\mu_{_U})$ is the closed support of the singular measure $\mu_{_U}$  associated with the singular part of $U.$

For a closed ideal $\cI$  of $\cL,$ we denote by
$U_{_\cI}$  the greatest inner common divisor of the inner parts of
functions in $\cI\setminus\{0\}$  and by $\E_{_{\cI}}$ the standard hull of $\cI$, that is
$$\E_{_{\cI}}:=\{\xi\in\T\text{ : }f(\xi)= 0 ,\ \forall f\in\cI\},$$
see \cite[Page. 85]{Hof} and \cite{Rud}.
We define
$$\cI_{\omega}(\E, U):=\{f\in\cL\mbox{ :
} f_{|\E}\equiv0\mbox{ and
}f/U\in\cH^{\infty}\},$$
and
$$\cJ_{\omega}(\E, U):=\Big\{f\in\cI_{\omega}(\E, U)\text{ : }
\lim_{\delta\rightarrow 0}\sup_{\xi,\zeta\in\E(\delta)\atop \xi\neq \zeta}\frac{|f(\xi)-f(\zeta)|}{\omega(|\xi-\zeta|)}=0\Big\},$$
where
$$\E(\delta):=\{\xi\in\T\ :\ d(\xi,\E)\leq\delta\} ,\qquad \delta>0,$$
and  $d(z,\E)$ designs the Euclidian distance from the point $z\in\overline{\D}$ to the set $\E.$
It should be mentioned that for any $\omega,$ the closed ideal $\cJ_{\omega}(\E, U)$ coincides with
$$\cK_{\omega}(\E, U):=\Big\{f\in\cI_{\omega}(\E, U)\ :\
\lim_{\delta\rightarrow 0}\sup_{d(z,\E),d(w,\E)\leq\delta\atop z,w\in\D \text{ and } z\neq w}\frac{|f(z)-f(w)|}{\omega(|z-w|)}=0\Big\},$$
as it is obtained in Proposition \ref{coincide} of section \ref{sect51} below. For $0<\alpha\leq1,$ we set
\begin{equation*}\label{f11}
\omega_\alpha(t):=t^\alpha, \qquad 0\leq t\leq 2.
\end{equation*}
A classical result due to Carleson \cite{Car} says that $\E$  is a boundary zero set of a function  $f\in Lip_{\omega_\alpha}\setminus\{0\}$ if and only if
\begin{equation}\label{f2}
\int_{\T} \log d(\xi,\E)|d\xi|>-\infty.
\end{equation}
In this case  $\E$ is called a Beurling-Carleson set. Recently, it is proved \cite[Theorem 1]{BZ} that if $\cI$ is a closed ideal of $Lip_{\omega_\alpha},$ then
\begin{equation}\label{f1}
\cJ_{\omega_\alpha}(\E_{_\cI}, U_{_\cI})\subseteq\cI.
\end{equation}
The proof of this result uses the following stronger property: If $\E$ is a Beurling-Carleson set then there exists a function $f$ holomorphic in $\D$ and continuously differentiable in $\overline{\D}$ such that both $f$ and $f'$ vanish exactly on $\E$.
On the other hand, Shirokov \cite[Chapter III]{Shi} has shown that generally the boundary zero set of a function  $f\in Lip_{\omega}$ does not necessarily satisfy \eqref{f2} as in the case of $Lip_{\chi_{_\beta}},$ where
\begin{equation*}\label{f3}
\chi_{_\beta}(t):=\log^{-\beta}\big(\frac{2e^{\beta}}{t}\big), \qquad 0< t\leq 2,
\end{equation*}
and $\beta>0$ is a fixed real number. According to these facts, we see that the method used in \cite[Proof of Theorem 1]{BZ} to prove \eqref{f1} does not work for $Lip_{\chi_{_\beta}}.$

A result of type \eqref{f1} was stated first by H. Hedenmalm \cite{Hed} in the algebras $\cH^{\infty}$ and $Lip_{\omega_1},$ for  closed ideals $\cI$  such that $\E_{_\cI}$ is a single point. Later, T. V. Pederson \cite{Ped} has obtained the same result in $Lip_{\omega_\alpha},$ for  closed ideals $\cI$  such that $\E_{_\cI}$ is countable. In this paper we are interested in the case where $\omega$ satisfies the following condition
\begin{equation}\label{cond}
\inf_{0< t\leq 1}\frac{\omega(t^2)}{\omega(t)}=:\eta_\omega>0.
\end{equation}

We obtain the following theorem which will be proved in the next section.

\begin{thm}\label{main1}
Let $\cI\subseteq\cL$ be a closed ideal, where $\omega$ is a modulus of continuity satisfying \eqref{cond}.
Then $\cJ_{\omega}(\E_{_\cI}, U_{_\cI})\subseteq\cI.$
\end{thm}

In particular Theorem \ref{main1} provides us with an extension of \eqref{f1} to a large class of smooth algebras between $\cH^\infty$ and $\bigcup\limits_{0<\alpha\leq 1} Lip_{\omega_\alpha},$ such as  $Lip_{\chi_{_\beta}}$ and $Lip_{\psi_{_\beta}},$ where
$$\psi_{_\beta}(t):=\log^{-\beta} \log\big(\frac{2e^{1+\beta}}{t}\big), \qquad 0< t\leq 2.$$
As consequence of Theorem \ref{main1} we deduce the following corollary.

\begin{coro}\label{corol}
We suppose that $\E$ or $U$ are nontrivial such that $\E\supseteq\sigma(U)\cap\T$ and
\begin{equation}\label{f4}
\int_{\T} \log \omega\big(d(\xi,\Z_U\cup\E)\big)|d\xi|>-\infty,
\end{equation}
where $\omega$ is a modulus of continuity satisfying \eqref{cond}. Then $\cJ_{\omega}(\E, U)$ is a nontrivial principal closed ideal with  $\E$ as hull and $U$ as inner factor.
\end{coro}

The estimate \eqref{f4} is in fact a necessary condition guaranteeing that $\cJ_{\omega}(\E, U)$ is nontrivial,
see the proof of Corollary \ref{corol}. By joining together the results established in Theorem \ref{main1} and Corollary \ref{corol} we deduce that $\cJ_{\omega}(\E, U)$ is the smallest closed ideal with $\E$ as hull and $U$ as inner factor provided that \eqref{f4} is satisfied.

To prove Theorem \ref{main1} we give an adaptation in the space $\cJ_{\omega}(\E):=\cJ_{\omega}(\E,1)$ of  Korenblum's functional approximation method, see for instance \cite{Bou1,Kor,Mat}. To do so, we will use some properties enjoyed by the space $\cJ_{\omega}(\E)$ which we describe in the following theorem.

\begin{thm}\label{main2}
Let $g\in \cJ_{\omega}(\E)$ be a function, where $\omega$ is an arbitrary modulus of continuity. Let $V\in\cH^{\infty}$ be an inner function dividing $g,$ that is $g/V\in\cH^{\infty}.$
Then
\begin{equation}\label{fprop1}
VO_{g}\in \cJ_{\omega}(\E),
 \end{equation}
 where $O_{g}$ is the outer factor of $g.$

Let $\{V_n\ :\ n\in\N\}\subset\cH^{\infty}$ be a sequence of inner functions dividing $g$ such that $\sigma(V_n)\cap\T\subseteq\E,$ for every $n\in\N.$ If
\begin{equation}\label{fprop1.5}\lim\limits_{
n\rightarrow\infty} \|V_nO_{g}-VO_{g}\|_{\infty}=0,
 \end{equation}
then
\begin{equation}\label{fprop2}
   \lim\limits_{n\rightarrow\infty} \|V_nO_{g}-VO_{g}\|_{\omega}=0.
 \end{equation}
\end{thm}
In particular,  assertion \eqref{fprop1} shows that $\cJ_{\omega}(\E)$ possesses the F-property.
This kind of factorization property has been proved by Havin, Shamoyan and Shirokov for the whole space $\cL$ with respect to any modulus of continuity,  see for instance \cite{HS}, \cite[Chapter I]{Shi} and also \cite[Appendix B]{Bou1}.

The remainder  of this paper is organized as follows: In section \ref{sect2}, we give the proof of Theorem \ref{main1}
and Corollary \ref{corol}. Section \ref{sect31} is devoted to presenting some needed results.
Sections \ref{sect4}  and \ref{sect3} contain respectively the proofs of  Propositions \ref{lem22} and \ref{lem2}. In the last section we give the proof of Theorem \ref{main2}.

\section{\bf Proofs of Theorem \ref{main1} and Corollary \ref{corol}.\label{sect2}}

In section \ref{sect42} we give the proof of the following proposition.

\begin{prop}\label{lem22}
Let  $\cI\subseteq\cL$ be  a closed ideal and let $g\in\cJ_{\omega}(\E_{_\cI})$ be a function, where $\omega$ is a modulus of continuity satisfying \eqref{cond}.
We suppose that $gf\in\cI$ for some outer function $f\in\cL.$ Then $g\in\cI.$
\end{prop}

For $f\in\cH^{\infty}$ we denote by $U_{f}$ and $O_{f}$  respectively the inner  and the outer factor of $f$. In section \ref{sect32} we will use both Theorem \ref{main2} and Proposition \ref{lem22} to give the proof of the following proposition.

\begin{prop}\label{lem2}
Let $\cI\subseteq\cL$ be a closed ideal, where $\omega$ is a modulus of continuity satisfying \eqref{cond}. There exists a function $f\in\cI$ such that $U_f=U_{_\cI}.$
\end{prop}

\subsection*{\bf \it Proof of Theorem \ref{main1}\label{sect21}}

Let $\cI\subseteq\cL$ be a closed ideal. By applying Proposition \ref{lem2}, there exists a function
$f\in\cI$ such that $U_f=U_{_{\cI}}.$ Let $g\in\cJ_\omega(\E_{_{\cI}}, U_{_{\cI}})$ be a function.
We note that $g/U_{_{\cI}}\in\cL,$ by using the F-property of $\cL.$
Since $gO_f=(g/U_{_{\cI}})f$ then $gO_f\in\cI.$  Thus $g\in\cI,$ by applying Proposition \ref{lem22}. Therefore $\cJ_\omega(\E_{_{\cI}}, U_{_{\cI}})\subseteq\cI,$ which is the desired result. $\hfill\qed$

\subsection*{\bf \it Proof of Corollary \ref{corol}\label{sect22}}

It is known \cite[Page. 137]{Shi} that \eqref{f4} ensure the existence of a function $f\in\cL$ such that $f^{-1}(0)=\Z_{U}\cup\E$
and $f/U\in\cH^\infty.$ The functions $O_f$ and $UO_f$ belong to $\cL,$  since $\cL$ possesses the F-property. Therefore $g:=UO^{2}_f\in\cJ_\omega(\E, U),$ as product of two functions in $\cI_\omega(\E):=\cI_\omega(\E, 1).$ Hence $\cJ_\omega(\E, U)$ is non trivial and possesses $\E$ as hull and $U$ as inner factor. Now, we let $[g]$ be the closed ideal in $\cL$ generated by $g.$ Clearly $[g]\subseteq\cJ_\omega(\E, U).$ Since $\E_{_{[g]}}=\E$ and $U_{_{[g]}}=U,$ then
$\cJ_\omega(\E, U)\subseteq [g],$ by applying Theorem \ref{main1}.
Therefore $\cJ_\omega(\E, U)=[g]$ and hence
$\cJ_\omega(\E, U)$ is principal.  The proof of Corollary \ref{corol} is completed. $\hfill\qed$

\section{\bf Approximation Lemmas. \label{sect31}}

We first note that $\omega$ is positive on $]0,2]$ and since $t\mapsto \omega(t)/t$ is nonincreasing,
\begin{equation}\label{omega}
\omega(2t)\leq 2\omega(t), \qquad t\in[0,1].
\end{equation}
For a closed subset $\E\subseteq\T,$ we let $\{(a_m,b_m)\ :\ m\in\N\}$ be a sequence of distinct arcs such that
$\T\setminus\E=\bigcup\limits_{m\in\N}(a_m,b_m),$ where
$(a_m,b_m)$ is an open arc of $\T\setminus\E$ joining the points $a_m,b_m\in\E.$ We set
$$\Omega_{\E}:=\{\bigcup_{m\in M}(a_m,b_m)\ :\ M\subseteq\N\}.$$
The next lemma will be used to simplify the following proofs of Lemma \ref{fliyo}, Lemma \ref{lem3}  and Theorem \ref{main2}.

\begin{lem}\label{hilbert}
Let $g\in\cI_{\omega}(\E)$ be a function, where $\omega$ is an arbitrary modulus of continuity.
Let $\{\Delta_n\in\Omega_{\E}\ :\ n\in\N\}$ be a sequence such that $\Delta_{n}\subseteq\Delta_{n+1},$ for every $n\in\N,$ and
$\bigcup\limits_{n\in\N}\Delta_n=\T\setminus\E.$
We consider a sequence of functions $\{h_n\in\cH^{\infty}\ :\ n\in\N\}$ such that;
\begin{itemize}
\item [$i_1.$] $\sup\limits_{n\in\N}\|h_n\|_{\infty}<+\infty,$
\item [$i_2.$] For each $n\in\N,$  the function $h_n$ has an extension both to a continuous function on $\overline{\D}\setminus\E$ and to an analytic function across $\Delta_n,$
\item [$i_3.$] For each $p\in\N,$ the both sequences $\{h_n\ :\ n\geq p\}$ and  $\{h^{'}_n\ :\ n\geq p\}$ converge uniformly to $0$ on each compact subset $\K\subseteq\Delta_p.$
\end{itemize}
Then $gh_n\in\cI_\omega(\E)$ if and only if
\begin{equation}\label{9amar1}
\sup_{\xi,\zeta\in\T\setminus\E\atop\xi\neq\zeta
}\Big|g(\xi)\frac{h_n(\xi)-h_n(\zeta)}{\omega(|\xi-\zeta|)}\Big|<+\infty,
\end{equation}
where  $n\in\N.$
If moreover we suppose that $g\in\cJ_{\omega}(\E)$ and
\begin{equation}\label{chams2}
\lim_{\delta\rightarrow 0}\sup_{\xi,\zeta\in\E(\delta)\setminus\E\atop\xi\neq\zeta
\text{ and }
|\xi-\zeta|\leq\frac{1}{2}d(\xi,\E)}\Big|g(\xi)\frac{h_n(\xi)-h_n(\zeta)}{\omega(|\xi-\zeta|)}\Big|=0,
\end{equation}
uniformly with respect to all numbers $n\in\N,$ then $gh_n\in\cJ_\omega(\E)$ for sufficiently large numbers $n\in\N,$ and
$\displaystyle\lim_{n\rightarrow \infty}\|gh_n\|_{\omega}=0.$

\end{lem}

\begin{proof}
By definition, a function  $g\in\cI_{\omega}(\E)$  is continuous on $\overline{\D}$ and vanishes on $\E.$
Let $\{h_n\in\cH^{\infty}\ :\ n\in\N\}$ be a sequence of functions satisfying the hypotheses $i_1$, $i_2$ and $i_3$ of the lemma.
Since  $h_n$ is bounded and has an extension to a continuous function on $\overline{\D}\setminus\E,$  the product $gh_n$ possesses an extension to a function in $\cA(\D)$ with vanishing values on $\E,$ for every $n\in\N.$ Using the facts that $h_n\in\cH^{\infty},$ $g\in\cL$ and the following equality
\begin{eqnarray}\label{9amar2}
\frac{gh_n(\xi)-gh_n(\zeta)}{\omega(|\xi-\zeta|)}
=g(\xi)\frac{h_n(\xi)-h_n(\zeta)}{\omega(|\xi-\zeta|)}+h_n(\zeta)\frac{g(\xi)-g(\zeta)}{\omega(|\xi-\zeta|)},
\end{eqnarray}
where $\xi,\zeta\in\T\setminus\E$ are two distinct points,  we see easily that $gh_n\in\cL$ if and only if \eqref{9amar1} holds.
Thus $gh_n\in\cI_\omega(\E)$ if and only if \eqref{9amar1} is satisfied.
We now suppose additionally that $g\in\cJ_{\omega}(\E).$ If we show both
\begin{equation}\label{chams00}
\lim_{n\rightarrow \infty}\sup_{\xi,\zeta\in\T\setminus\E(\delta)\atop \xi\neq\zeta}\frac{|gh_n(\xi)-gh_n(\zeta)|}{\omega(|\xi-\zeta|)}=0
\end{equation}
for each $\delta\in]0,1],$ and
\begin{equation}\label{chams11}
\lim_{\delta\rightarrow 0}\sup_{\xi,\zeta\in\E(\delta)\atop\xi\neq\zeta}\frac{|gh_n(\xi)-gh_n(\zeta)|}{\omega(|\xi-\zeta|)}=0
\end{equation}
uniformly with respect to $n\in\N,$ then  $gh_n\in\cJ_{\omega}(\E),$ for sufficiently large numbers  $n\in\N,$ and $\displaystyle\lim_{n\rightarrow \infty}\|gh_n\|_{\omega}=0.$
Indeed, we suppose that \eqref{chams00} and \eqref{chams11} are satisfied.
Let $\varepsilon$ be a positive number. Using \eqref{chams11}, there exists  $\delta_\varepsilon\in]0,1]$  such that
\begin{equation}\label{chams31}
\sup_{\xi,\zeta\in\E(\delta_\varepsilon)\atop\xi\neq\zeta}\frac{|gh_n(\xi)-gh_n(\zeta)|}{\omega(|\xi-\zeta|)}\leq\varepsilon,\qquad \text{ for every }n\in\N.
\end{equation}
We observe that $\overline{\T\setminus\E(\delta_\varepsilon)}\subseteq \T\setminus\E(\delta_\varepsilon/2).$ Then,
by using \eqref{chams00}, there exists $n_\varepsilon\in\N$ such that
\begin{equation}\label{chams30}
\sup_{\xi,\zeta\in\overline{\T\setminus\E(\delta_\varepsilon)}\atop \xi\neq\zeta}\frac{|gh_n(\xi)-gh_n(\zeta)|}{\omega(|\xi-\zeta|)}\leq\varepsilon,\qquad  \text{ for every } n\geq n_\varepsilon.
\end{equation}
On the other hand, for two points $\xi_{_0}\in\E(\delta_\varepsilon)$ and $\zeta_{_0}\in\T\setminus\E(\delta_\varepsilon)$ we can associate a third point $z_{\xi_{_0},\zeta_{_0}}$ in the boundary of $\E(\delta_\varepsilon)$ such that $$\max\{|\xi_{_0}-z_{\xi_{_0},\zeta_{_0}}|,|z_{\xi_{_0},\zeta_{_0}}-\zeta_{_0}|\}\leq |\xi_{_0}-\zeta_{_0}|.$$
Then,
\begin{eqnarray}\label{chams4}
\frac{|gh_n(\xi_{_0})-gh_n(\zeta_{_0})|}{\omega(|\xi_{_0}-\zeta_{_0}|)}
 \leq \frac{|gh_n(\xi_{_0})-gh_n(z_{\xi_{_0},\zeta_{_0}})|}{\omega(|\xi_{_0}-z_{\xi_{_0},\zeta_{_0}}|)}+ \frac{|gh_n(z_{\xi_{_0},\zeta_{_0}})-gh_n(\zeta_{_0})|}{\omega(|z_{\xi_{_0},\zeta_{_0}}-\zeta_{_0}|)},
\end{eqnarray}
since $\omega$ is nondecreasing.
Combining the estimates \eqref{chams31}, \eqref{chams30} and \eqref{chams4}, we obtain
\begin{eqnarray}\label{chams5}
\sup_{\xi,\zeta\in\T\setminus\E\atop \xi\neq\zeta}\frac{|gh_n(\xi)-gh_n(\zeta)|}{\omega(|\xi-\zeta|)}
\leq 2\varepsilon,\qquad  \text{ for every } n\geq n_\varepsilon.
\end{eqnarray}
It follows $gh_n\in\cI_\omega(\E),$  for every $n\geq n_\varepsilon,$ and hence $gh_n\in\cJ_\omega(\E),$ since
we have assumed \eqref{chams11}. We now observe that for each $\delta\in]0,1],$
there exists $n_\delta\in\N$ such that
$\overline{\T\setminus\E(\delta)}\subseteq\Delta_n,$ for all numbers $n\geq n_\delta,$ since  $\bigcup\limits_{n\in\N}\Delta_n=\T\setminus\E.$
Then, for each $\delta\in]0,1],$ the sequence $\{h_n\ :\ n\in\N\}$ converges uniformly to $0$ on $\overline{\T\setminus\E(\delta)},$
 by using the hypothesis $i_3.$  Furthermore, by using hypothesis $i_1,$
 $$\lim_{\delta\rightarrow0}\sup_{\xi\in\E(\delta)\atop n\in\N}|gh_n(\xi)|=0.$$
We deduce  $\displaystyle\lim_{n\rightarrow \infty}\|gh_n\|_{\infty}=0.$
Hence $\displaystyle\lim_{n\rightarrow \infty}\|gh_n\|_{\omega}=0,$ by using once more \eqref{chams5}.

So, it remains to show that  \eqref{chams00} and \eqref{chams11} are fulfilled if   \eqref{chams2} holds. For this we fix two distinct points $\xi,\zeta\in\T\setminus\E$  and a number $\delta\in]0,1].$
We consider the following cases:

\textit{Case $A.$ } We suppose that $\displaystyle\xi,\zeta\in\T\setminus\E(\delta).$

\textit{\hspace{0.5cm}$A_1.$} If $\displaystyle|\xi-\zeta|\geq\frac{1}{2}d(\xi,\E),$ then
$$\omega(|\xi-\zeta|)\geq \omega(\frac{1}{2}d(\xi,\E))\geq \frac{1}{2}\omega(d(\xi,\E)),$$
by using the properties of $\omega.$ Thus
\begin{eqnarray}\label{radab3}
\Big|g(\xi)\frac{h_n(\xi)-h_n(\zeta)}{\omega(|\xi-\zeta|)}\Big|
\leq 4\|g\|_{\omega} \sup_{z\in\T\setminus\E(\delta)}|h_n(z)|, \qquad  n\geq n_\delta.
\end{eqnarray}

\textit{\hspace{0.5cm}$A_2.$} In the case where $\displaystyle|\xi-\zeta|\leq\frac{1}{2}d(\xi,\E),$  we consider the open arc $(\xi,\zeta)$ of $\T$ joining  the points $\xi$ and $\zeta,$ and such that $(\xi,\zeta)\subseteq\T\setminus\E(\delta).$ In all what follows we use the notation
$X\lesssim Y$ to mean that there exists a positive constant $c$ such that $X\leq cY,$  where $X$ and $Y$ are two nonnegative functions.
Since $t\mapsto t/\omega(t)$ is nondecreasing,
\begin{eqnarray}\nonumber
\Big|g(\xi)\frac{h_n(\xi)-h_n(\zeta)}{\omega(|\xi-\zeta|)}\Big|
&\leq& |g(\xi)|\frac{|h_n(\xi)-h_n(\zeta)|}{|\xi-\zeta|}\frac{|\xi-\zeta|}{\omega(|\xi-\zeta|)}
\\\nonumber&\lesssim& \frac{2\|g\|_{\infty}}{\omega(2)}\sup_{z\in(\xi,\zeta)}|h_{n}^{'}(z)|
\\\label{radab4}&\lesssim& \|g\|_{\infty}\sup_{z\in\T\setminus\E(\delta)}|h_{n}^{'}(z)|, \qquad  n\geq n_\delta.
\end{eqnarray}
From \eqref{radab3} and \eqref{radab4} we deduce that for each $\delta\in]0,1],$ there exists $n_\delta\in\N$ such that
 \begin{eqnarray}\label{radab5}
\sup_{\xi,\zeta\in\T\setminus\E(\delta)\atop \xi\neq\zeta}\Big|g(\xi)\frac{h_n(\xi)-h_n(\zeta)}{\omega(|\xi-\zeta|)}\Big|
\lesssim \|g\|_{\omega}\big( \sup_{z\in\T\setminus\E(\delta)}|h_n(z)|+\sup_{z\in\T\setminus\E(\delta)}|h_{n}^{'}(z)| \big),
\end{eqnarray}
for all numbers  $n\geq n_\delta.$
On the other hand it is plain that
 \begin{eqnarray}\label{radab5'}
\sup_{\xi,\zeta\in\T\setminus\E(\delta)\atop \xi\neq\zeta}\Big|h_n(\zeta)\frac{g(\xi)-g(\zeta)}{\omega(|\xi-\zeta|)}\Big|
\lesssim \|g\|_{\omega} \sup_{z\in\T\setminus\E(\delta)}|h_n(z)|, \qquad n\in\N.
\end{eqnarray}
Taking account of the hypothesis  $i_3,$  the equality \eqref{chams00} is deduced by combining together \eqref{9amar2}, \eqref{radab5} and \eqref{radab5'}.

\textit{Case $B.$ } Now, we suppose that $\displaystyle\xi,\zeta\in\E(\delta)\setminus\E.$
If we suppose additionally that $\displaystyle|\xi-\zeta|\geq\frac{1}{2}d(\xi,\E),$ then
\begin{eqnarray}\nonumber
\Big|g(\xi)\frac{h_n(\xi)-h_n(\zeta)}{\omega(|\xi-\zeta|)}\Big|
&\leq& 4\sup_{n\in\N}\|h_n\|_{\infty} \sup_{\xi\in\E(\delta)}\frac{|g(\xi)|}{\omega(d(\xi,\E))}
\\\label{radab8}&=&o(1),\qquad \text{ as }\delta\rightarrow0,
\end{eqnarray}
by using the hypothesis $i_1$ and the fact that  $g\in\cJ_\omega(\E).$
It follows
\begin{equation}\label{radab9}
\lim_{\delta\rightarrow 0}\sup_{\xi,\zeta\in\E(\delta)\setminus\E\atop\xi\neq\zeta
\text{ and }
|\xi-\zeta|\geq\frac{1}{2}d(\xi,\E)}\Big|g(\xi)\frac{h_n(\xi)-h_n(\zeta)}{\omega(|\xi-\zeta|)}\Big|=0,
\end{equation}
and hence
\begin{eqnarray}\label{radab10}
\lim_{\delta\rightarrow 0}\sup_{\xi,\zeta\in\E(\delta)\setminus\E\atop\xi\neq\zeta}\Big|g(\xi)\frac{h_n(\xi)-h_n(\zeta)}{\omega(|\xi-\zeta|)}\Big|=0,
\end{eqnarray}
uniformly with respect to $n\in\N,$ since we have assumed \eqref{chams2}.
On the other hand, by using again the hypothesis $i_1$ and the fact that $g\in\cJ_\omega(\E),$ it is easily seen that
 \begin{eqnarray}\label{radab10'}
\lim_{\delta\rightarrow 0}\sup_{\xi,\zeta\in\E(\delta)\setminus\E\atop\xi\neq\zeta}\Big|h_n(\zeta)\frac{g(\xi)-g(\zeta)}{\omega(|\xi-\zeta|)}\Big|
=0,
\end{eqnarray}
uniformly with respect to $n\in\N.$
Hence, the equality \eqref{chams11} is deduced by combining together  \eqref{9amar2}, \eqref{radab10} and \eqref{radab10'}.
The proof of Lemma \ref{hilbert} is completed.
\end{proof}

Let $f\in\cA(\D)$ be an outer function and let $\Gamma\in\Omega_{_{\E}},$ for some closed subset $\E$ of $\T.$ We define the following Korenblum's outer function
\begin{equation}\label{korenblum}
f_{\Gamma}(z):=\exp\Big\{\frac{1}{2\pi}\int_{\Gamma}\frac{\xi+z}{\xi-z}\log |f|(\xi)|d\xi|\Big\},\qquad z\in\D.
\end{equation}
We observe that $f_{\Gamma}$ belongs to $\cH^\infty$ with boundary values satisfying
$$|f_\Gamma|(\xi)=\left\{
  \begin{array}{ll}
   |f|(\xi), \quad \ \hbox{ if } \xi\in\Gamma,  \\
    1, \qquad \quad \ \   \hbox{if } \xi\in\T\setminus\overline{\Gamma}.
  \end{array}
\right.$$
We also observe that $f_\Gamma=f\times f^{-1}_{\T\setminus\Gamma}$ and that $f_{\Gamma}$ has an extension  both to an analytic function on $\C\setminus\overline{\Gamma}$ and to a continuous one on $\overline{\D}\setminus\E.$ For a function $f\in\cA(\D),$
we denote by $\E_f$ the boundary zero set of $f.$
The following lemma will be used in the proofs of Proposition \ref{lem22} and  Proposition \ref{lem2}.

\begin{lem}\label{fliyo}
Let $f\in\cL$ be an outer function, where $\omega$ is a modulus of continuity satisfying \eqref{cond}. Let $g\in\cL$ be a function and let $\E$ be a closed subset of $\E_g.$ Then  $g f_{{\Gamma}}\in\cL$ and
\begin{equation}\label{fig11}
\|g f_{{\Gamma}}\|_\omega\leq c_{\omega,f}\|g\|_\omega, \qquad \text{ for every }\Gamma\in\Omega_{_{\E}},
\end{equation}
where $c_{\omega,f}>0$ is a constant not depending on both $\Gamma$ and $g.$
If moreover $g\in\cJ_{\omega}(\E),$  then
\begin{equation}\label{fig1}
g f_{{\Gamma}}\in\cJ_{\omega}(\E), \qquad \text{ for every }\Gamma\in\Omega_{_{\E}},
\end{equation}
and
\begin{equation}\label{fig2}
\lim\limits_{n\rightarrow \infty}\|g f_{\Gamma_n}-g\|_{\omega}=0,
\end{equation}
where $\Gamma_{n}:=\bigcup\limits_{m> n}(a_m,b_m),$ for each $n\in\N.$
\end{lem}

\begin{proof} Let $g\in\cL$ be a function. Let $\Gamma\in\Omega_{_{\E}},$ where  $\E$ is a fixed closed subset of $\E_g.$  Since $f_{{\Gamma}}\in\cH^\infty$ and has a continuous extension on $\overline{\D}\setminus\E,$
then $gf_{{\Gamma}}\in\cA(\D)$ with vanishing values on $\E.$
Let $\xi,\zeta\in\T\setminus\E$ be two distinct points such that $d(\xi,\E)\leq d(\zeta,\E).$
Two cases are possible;

\textit{Case $A.$} We suppose that
 $\displaystyle|\xi-\zeta|\geq \frac{1}{4}d^2(\xi,\E).$ Since $\omega$ satisfies \eqref{cond},
\begin{eqnarray}\nonumber
\Big|g(\xi)\frac{f_{\Gamma}(\xi)-f_{\Gamma}(\zeta)}{\omega(|\xi-\zeta|)}\Big|&\leq&
2\max\{1,\|f\|_{\infty}\}\frac{|g(\xi)|}{\omega(|\xi-\zeta|)}
\\&\leq&\label{bi1} 4\eta_{\omega}^{-1}\max\{1,\|f\|_{\infty}\}\frac{|g(\xi)|}{\omega(d(\xi,\E))}
\\&\leq&\label{bi11} 4\eta_{\omega}^{-1}\max\{1,\|f\|_{\infty}\}\|g\|_{\omega}.
\end{eqnarray}

\textit{Case $B.$} We suppose that
 $\displaystyle|\xi-\zeta|\leq \frac{1}{4}d^2(\xi,\E).$ In this case, we let $(\xi,\zeta)$ be the arc such that
 $(\xi,\zeta)\subseteq\T\setminus\E.$ We observe that
$$d(z,\E)\geq d(\xi,\E), \qquad \text{ for all }z\in(\xi,\zeta).$$

 \textit{\hspace{0.5cm}$B_1.$} We assume that $\displaystyle(\xi,\zeta)\subseteq \T\setminus\overline{\Gamma}.$
We simply compute
$$ \Big|\frac{f_{\Gamma}(\xi)-f_{\Gamma}(\zeta)}{\xi-\zeta}\Big|\lesssim \sup_{z\in(\xi,\zeta)}|f^{'}_{\Gamma}(z)|\leq\frac{c_{_f}}{d^2(\xi,\E)},$$
where $c_f>0$ is a constant not depending  on both $\Gamma$ and the points $\xi$ and $\zeta.$
Since $t\mapsto t/\omega(t)$ is nondecreasing, $\omega$ satisfies \eqref{cond} and $\displaystyle|\xi-\zeta|\leq \frac{1}{4}d^2(\xi,\E),$ we get
\begin{eqnarray}\nonumber
\Big|g(\xi)\frac{f_{\Gamma}(\xi)-f_{\Gamma}(\zeta)}{\omega(|\xi-\zeta|)}\Big|&=&
\frac{|g(\xi)||\xi-\zeta|}{\omega(|\xi-\zeta|)}\frac{|f_{\Gamma}(\xi)-f_{\Gamma}(\zeta)|}{|\xi-\zeta|}
\\\nonumber&\leq&c_f\frac{|g(\xi)||\xi-\zeta|}{\omega(|\xi-\zeta|)d^2(\xi,\E)}
\\\label{bi3}&\leq&\frac{c_f\eta^{-1}_{\omega}}{2}\frac{|g(\xi)|}{\omega(d(\xi,\E))}
\\\label{bi33}&\leq&\frac{c_f\eta^{-1}_{\omega}}{2}\|g\|_{\omega}.
\end{eqnarray}

\textit{\hspace{0.5cm} $B_2.$} We now assume that $\displaystyle(\xi,\zeta)\subseteq \Gamma.$
We have  $|f_{\T\setminus\Gamma}(\xi)|=|f_{\T\setminus\Gamma}(\zeta)|=1.$ The inequality \eqref{bi3} and the following equality
\begin{eqnarray}\nonumber
f_{\Gamma}(\xi)-f_{\Gamma}(\zeta)&=&
f(\xi)f^{-1}_{\T\setminus\Gamma}(\xi)-f(\zeta)f^{-1}_{\T\setminus\Gamma}(\zeta)\\\nonumber&=&
f^{-1}_{\T\setminus\Gamma}(\xi)\big(f(\xi)-f(\zeta)\big)+f(\zeta)\big(f^{-1}_{\T\setminus\Gamma}(\xi)-f^{-1}_{\T\setminus\Gamma}(\zeta)\big)\\\nonumber&=&
f^{-1}_{\T\setminus\Gamma}(\xi)\big(f(\xi)-f(\zeta)\big)\\\label{revien}
&&-f(\zeta)f^{-1}_{\T\setminus\Gamma}(\xi)f^{-1}_{\T\setminus\Gamma}(\zeta)
\big(f_{\T\setminus\Gamma}(\xi)-f_{\T\setminus\Gamma}(\zeta)\big)
\end{eqnarray}
give the following estimate
\begin{eqnarray}\nonumber
&&\Big|g(\xi)\frac{f_{\Gamma}(\xi)-f_{\Gamma}(\zeta)}{\omega(|\xi-\zeta|)}\Big|
\\\nonumber&\leq&
|g(\xi)|\frac{|f(\xi)-f(\zeta)|}{\omega(|\xi-\zeta|)}
+|f(\zeta)||g(\xi)|\frac{|f_{\T\setminus\Gamma}(\xi)-f_{\T\setminus\Gamma}(\zeta)|}{\omega(|\xi-\zeta|)}
\\\label{bi4}&\leq&
|g(\xi)|\|f\|_{\omega}+ \frac{c_f\eta^{-1}_{\omega}}{2}\|f\|_{\infty}\frac{|g(\xi)|}{\omega(d(\xi,\E))}
\\\label{bi44}&\leq&(1+\frac{c_f\eta^{-1}_{\omega}}{2})\|g\|_{\omega}\|f\|_{\omega}.
\end{eqnarray}
Joining together the estimates \eqref{bi11}, \eqref{bi33} and \eqref{bi44}, we obtain
\begin{eqnarray*}
\sup_{\xi,\zeta\in\T\setminus\E \atop \xi\neq\zeta}\Big|g(\xi)\frac{f_{\Gamma}(\xi)-f_{\Gamma}(\zeta)}{\omega(|\xi-\zeta|)}\Big|
\leq C_{\omega,f} \|g\|_{\omega},\qquad \Gamma\in\Omega_{_{\E}},
\end{eqnarray*}
where $C_{\omega,f}>0$ is a constant not depending on  $\Gamma$ and $g.$ On the other hand, it is obvious that
\begin{eqnarray*}
\sup_{\xi,\zeta\in\T\setminus\E \atop \xi\neq\zeta}\Big|f_{\Gamma}(\zeta)\frac{g(\xi)-g(\zeta)}{\omega(|\xi-\zeta|)}\Big|
\leq \|f\|_{\infty}\|g\|_{\omega},\qquad \Gamma\in\Omega_{_{\E}}.
\end{eqnarray*}
Therefore
\begin{eqnarray*}
\sup_{\xi,\zeta\in\T\setminus\E \atop \xi\neq\zeta}\Big|\frac{gf_{\Gamma}(\xi)-gf_{\Gamma}(\zeta)}{\omega(|\xi-\zeta|)}\Big|
\leq \big(C_{\omega,f}+\|f\|_{\infty}\big) \|g\|_{\omega},\qquad \Gamma\in\Omega_{_{\E}},
\end{eqnarray*}
which proves \eqref{fig11}, and so $gf_\Gamma\in\cL.$

We now suppose additionally that $g\in\cJ_{\omega}(\E).$ Then, by using  \eqref{bi1}, \eqref{bi3} and \eqref{bi4},
\begin{eqnarray}\label{atay2}
\lim_{\delta\rightarrow 0}\sup_{\xi,\zeta\in\E(\delta)\setminus\E\atop\xi\neq\zeta}\Big|g(\xi)\frac{f_{\Gamma}(\xi)-f_{\Gamma}(\zeta)}{\omega(|\xi-\zeta|)}\Big|=0,
\end{eqnarray}
uniformly with respect to all $\Gamma\in\Omega_{_{\E}}.$ On the other hand,  it is plain that
\begin{eqnarray}\label{atay3}
\lim_{\delta\rightarrow 0}\sup_{\xi,\zeta\in\E(\delta)\setminus\E\atop\xi\neq\zeta}\Big|f_\Gamma(\xi)\frac{g(\xi)-g(\zeta)}{\omega(|\xi-\zeta|)}\Big|=0,
\end{eqnarray}
uniformly with respect to all $\Gamma\in\Omega_{_{\E}},$ since we have supposed $g\in\cJ_{\omega}(\E).$
Therefore  $gf_{\Gamma}\in\cJ_{\omega}(\E).$
Taking account of the estimate \eqref{atay2}, we can now apply Lemma \ref{hilbert} with $\Delta_n=\T\setminus\{\overline{\Gamma_n}\cup\E\}$ and $h_n=1-f_{\Gamma_n}$
to deduce \eqref{fig2}.
This finishes the proof of Lemma \ref{fliyo}.
\end{proof}

\section{\bf Proof of Proposition \ref{lem22}. \label{sect4}}

The proof of the proposition will be given in subsection \ref{sect42}. Before this we need to point out some technical results.
\subsection{Technical Lemmas\label{sect41}}

We start with the following classical lemma.

\begin{lem}\label{lem3}
Let $g\in\cJ_{\omega}(\E)$ be a function, where $\omega$ is an arbitrary modulus of continuity.
Let $A:=\{a_k\ :\ 0\leq k\leq K\}$ be a finite sequence of points in $\E$ not necessarily distinct, where  $K\in\N.$   Then
\begin{equation}\label{shfang}
\lim\limits_{\rho\rightarrow 0^+}\|\varphi_{\rho,K}g-g\|_{\omega}=0,
\end{equation}
 where
$$\varphi_{\rho,K}(z):=\prod_{k=0}^{k=K}\frac{z\overline{a_k}-1}{z\overline{a_k}-1-\rho},\qquad z\in\D,$$
and $\rho\leq 1$ is a positive number.
\end{lem}

\begin{proof}
We will apply Lemma \ref{hilbert} with $\Delta_n=\T\setminus\E,$ for every $n\in\N.$  To prove \eqref{shfang} it is sufficient to show that
\begin{equation}\label{radab2}
\lim_{\delta\rightarrow 0}\sup_{\xi,\zeta\in\E(\delta)\setminus\E\atop\xi\neq\zeta\text{ and }
|\xi-\zeta|\leq\frac{1}{2}d(\xi,\E)}\Big|g(\xi)\frac{h_{\rho}(\xi)-h_{\rho}(\zeta)}{\omega(|\xi-\zeta|)}\Big|=0,
\end{equation}
uniformly with respect to all positive numbers $\rho\leq 1,$ where $h_{\rho}:=1-\varphi_{\rho,K}.$ For this, we let $\xi,\zeta\in\E(\delta)\setminus\E$ be two distinct points such that $|\xi-\zeta|\leq\frac{1}{2}d(\xi,\E).$ We consider the arc $(\xi,\zeta)$ satisfying $(\xi,\zeta)\subseteq\T\setminus\E.$ A simple calculation gives
$$d(z,\E)\geq \frac{1}{2}d(\xi,\E), \qquad \text{ for all }z\in(\xi,\zeta),$$
and
$$|\varphi^{'}_{\rho,K}(z)|\leq \frac{K}{d(z,\E)},\qquad z\in\T\setminus\E.$$
Since $t\mapsto t/\omega(t)$ is nondecreasing, we obtain
\begin{eqnarray}\nonumber
\Big|g(\xi)\frac{h_{\rho}(\xi)-h_{\rho}(\zeta)}{\omega(|\xi-\zeta|)}\Big|&=&
|g(\xi)|\frac{|\varphi_{\rho,K}(\xi)-\varphi_{\rho,K}(\zeta)|}{|\xi-\zeta|}\frac{|\xi-\zeta|}{\omega(|\xi-\zeta|)}
\\\nonumber&\lesssim&  |g(\xi)|\sup_{z\in(\xi,\zeta)}|\varphi^{'}_{\rho,K}(z)|\frac{|\xi-\zeta|}{\omega(|\xi-\zeta|)}
\\\nonumber&\leq& 2K \frac{|g(\xi)|}{d(\xi,\E)}\frac{|\xi-\zeta|}{\omega(|\xi-\zeta|)}
\\\label{eq99}&\leq& 2K \frac{|g(\xi)|}{\omega(d(\xi,\E))}.
\end{eqnarray}
Since by hypothesis $g\in\cJ_{\omega}(\E),$  the desired result \eqref{radab2} follows from \eqref{eq99}.
\end{proof}

The next lemma will be used to simplify the proof of the following one.

\begin{lem}\label{hamam}
Let $f\in\cL$ be an outer function, where $\omega$ is a modulus of continuity satisfying \eqref{cond}.
Let  $\gamma:=(a,b)$ be an open arc joining two points $a,b\in \E_{f}$ such that $\gamma\subseteq\T\setminus \E_{f}.$
Let $\varepsilon$ be a nonzero real number such that  $\gamma_{\varepsilon}:=(a e^{i\varepsilon},b e^{-i\varepsilon})\subset\gamma.$
Then
\begin{equation*}
\lim\limits_{\varepsilon\rightarrow0}\|\phi^{2}_{{\rho,0}}\phi^{}_{{\rho,\varepsilon}}
f_{{\gamma\setminus\gamma_{\varepsilon}}}-\phi^{3}_{{\rho,0}}\|_{\omega}=0,
\end{equation*}
where $\rho$ is a fixed positive number and
$$\displaystyle  \phi_{{\rho,\varepsilon}}(z):=\Big(\frac{z\overline{a}e^{-i\varepsilon}-1}
{z\overline{a}e^{-i\varepsilon}-1-\rho}\Big)
\Big(\frac{z\overline{b}e^{i\varepsilon}-1}{z\overline{b}e^{i\varepsilon}-1-\rho}\Big),\qquad z\in\D.$$
\end{lem}

\begin{proof}
Using \eqref{fig11},
\begin{equation}\label{figure0}
\|\phi^{}_{{\rho,0}} \phi^{}_{{\rho,\varepsilon}}f_{\gamma\setminus\gamma_{\epsilon}}\|_\omega\leq c_{\omega,f}
\|\phi^{}_{{\rho,0}}\|_\omega \|\phi^{}_{{\rho,\varepsilon}}\|_\omega,
\end{equation}
where $c_{\omega,f}>0$ is a constant not depending on $\varepsilon.$ Clearly
\begin{equation}\label{figure00}
\|\phi^{}_{{\rho,\varepsilon}}\|_\omega\leq \Big\|\frac{z\overline{a}e^{-i\varepsilon}-1}
{z\overline{a}e^{-i\varepsilon}-1-\rho} \Big\|_\omega\Big\|\frac{z\overline{b}e^{i\varepsilon}-1}{z\overline{b}e^{i\varepsilon}-1-\rho}\Big\|_\omega= \Big\|\frac{z-1}{z-1-\rho}\Big\|_\omega^2.
\end{equation}
Thus
\begin{equation}\label{figure1}
\|\phi^{}_{{\rho,0}} \phi^{}_{{\rho,\varepsilon}}f_{\gamma\setminus\gamma_{\epsilon}}\|_\omega\leq c_{\omega,f}
\Big\|\frac{z-1}{z-1-\rho}\Big\|_\omega^{4} .
\end{equation}
We set $\F:=\{a,b\}.$
We have $\phi^{2}_{{\rho,0}}\phi^{}_{{\rho,\varepsilon}}f_{\gamma\setminus\gamma_{\epsilon}}\in\cJ_\omega(\F),$ as product of two functions
 $\phi^{}_{{\rho,0}}$ and $\phi^{}_{{\rho,0}}\phi^{}_{{\rho,\varepsilon}}f_{\gamma\setminus\gamma_{\epsilon}}$ from $\cI_\omega(\F).$
Taking account of  \eqref{figure00}, \eqref{figure1} and the fact that $\| \phi^{}_{{\rho,\varepsilon}}\|_{\infty}\leq 1,$ we calculate, for all distinct points $\xi,\zeta\in\T,$ that
\begin{eqnarray}\nonumber
&&\frac{|\phi^{2}_{{\rho,0}}\phi^{}_{{\rho,\varepsilon}}f_{\gamma\setminus\gamma_{\epsilon}}(\xi)-
\phi^{2}_{{\rho,0}}\phi^{}_{{\rho,\varepsilon}}f_{\gamma\setminus\gamma_{\epsilon}}(\zeta)|}{\omega(|\xi-\zeta|)}
\\\nonumber&\leq&|\phi_{{\rho,0}}\phi^{}_{{\rho,\varepsilon}}f_{\gamma\setminus\gamma_{\epsilon}}(\xi)|
\frac{|\phi_{{\rho,0}}(\xi)-\phi_{{\rho,0}}(\zeta)|}{\omega(|\xi-\zeta|)}
+|\phi_{{\rho,0}}(\zeta)|\frac{|\phi_{{\rho,0}}\phi^{}_{{\rho,\varepsilon}}f_{\gamma\setminus\gamma_{\epsilon}}(\xi)-
\phi_{{\rho,0}}\phi^{}_{{\rho,\varepsilon}}f_{\gamma\setminus\gamma_{\epsilon}}(\zeta)|}{\omega(|\xi-\zeta|)}
\\\nonumber&\leq& |\phi_{{\rho,0}}(\xi)|(\|f\|_{\infty}+1)\|\phi_{{\rho,0}}\|_\omega
+|\phi_{{\rho,0}}(\zeta)|
\|\phi^{}_{{\rho,0}} \phi^{}_{{\rho,\varepsilon}}f_{\gamma\setminus\gamma_{\epsilon}}\|_\omega
\\\label{figure2}&\leq& c_{\omega,f,\rho}\big(|\phi_{{\rho,0}}(\xi)|
+|\phi_{{\rho,0}}(\zeta)|\big),
\end{eqnarray}
where $c_{\omega,f,\rho}>0$ is a constant not depending on $\varepsilon.$
Then
\begin{equation}\label{figure2}
\lim_{\delta\rightarrow0}\sup_{\xi,\zeta\in\F(\delta)\atop \xi\neq\zeta}\frac{|\phi^{2}_{{\rho,0}}\phi^{}_{{\rho,\varepsilon}}f_{\gamma\setminus\gamma_{\epsilon}}(\xi)-
\phi^{2}_{{\rho,0}}\phi^{}_{{\rho,\varepsilon}}f_{\gamma\setminus\gamma_{\epsilon}}(\zeta)|}{\omega(|\xi-\zeta|)}=0,
\end{equation}
and since $\phi^{3}_{{\rho,0}} \in\cJ_\omega(\F),$ we get
\begin{equation}\label{fig3}
\lim_{\delta\rightarrow0}\sup_{\xi,\zeta\in\F(\delta)\atop \xi\neq\zeta}\frac{|k_\varepsilon(\xi)-k_\varepsilon(\zeta)|}{\omega(|\xi-\zeta|)}=0,
\end{equation}
uniformly with respect to $\varepsilon,$ where  $k_\varepsilon:=\phi^{2}_{{\rho,0}}\phi^{}_{{\rho,\varepsilon}}
f_{{\gamma\setminus\gamma_{\varepsilon}}}-\phi^{3}_{{\rho,0}}.$

We will now show that
\begin{equation}\label{fig6}
\lim_{\varepsilon\rightarrow0}\sup_{\xi,\zeta\in\T\setminus\F(\delta)\atop\xi\neq\zeta}\frac{|k_\varepsilon(\xi)-k_\varepsilon(\zeta)|}{\omega(|\xi-\zeta|)}=0,
\end{equation}
for each number $\delta\in]0,1].$ We first remark that the function $k_\varepsilon$ has an analytic extension across  $\T\setminus\F(\delta),$ for  sufficiently small numbers $\varepsilon.$
A simple calculation gives
\begin{equation}\label{fig7}
\lim_{\varepsilon\rightarrow0}\sup_{z\in\T\setminus\F(\delta)}|k_\varepsilon(z)|=\lim_{\varepsilon\rightarrow0}\sup_{z\in\T\setminus\F(\delta)}|k^{'}_\varepsilon(z)|=0.
\end{equation}
For two distinct points $\xi,\zeta\in\T\setminus\F(\delta),$  we clearly have two possible cases:
$$\text{\it Case 1.}\quad |\xi-\zeta|\leq\frac{1}{2}d(\xi,\F)\qquad  \text{ or }\qquad \text{\it Case 2.}\quad  |\xi-\zeta|\geq\frac{1}{2}d(\xi,\F).$$
In the first case, by using the fact that $t\mapsto t/\omega(t)$ is nondecreasing
\begin{equation}\label{fig4}
\frac{|k_\varepsilon(\xi)-k_\varepsilon(\zeta)|}{\omega(|\xi-\zeta|)}\lesssim \sup_{z\in(\xi,\zeta)}|k^{'}_\varepsilon(z)|\frac{|\xi-\zeta|}{\omega(|\xi-\zeta|)}
\lesssim \sup_{z\in\T\setminus\F(\delta)}|k^{'}_\varepsilon(z)|,
\end{equation}
where  $(\xi,\zeta)$ is the open arc joining the points $\xi$ and $\zeta,$ and such that $(\xi,\zeta)\subseteq \T\setminus\F(\delta).$
In the second case, by using the fact that $\omega$ is nondecreasing and \eqref{omega}, we get
\begin{equation}\label{fig5}
\frac{|k_\varepsilon(\xi)-k_\varepsilon(\zeta)|}{\omega(|\xi-\zeta|)}
\leq \displaystyle 2\frac{\sup\limits_{z\in\T\setminus\F(\delta)}|k_\varepsilon(z)| }{\omega(\frac{1}{2}d(\xi,\F))} \leq\frac{4}{\omega(\delta)}\sup_{z\in\T\setminus\F(\delta)}|k_\varepsilon(z)|.
\end{equation}
Hence the desired equality \eqref{fig6} is deduced from \eqref{fig7}, \eqref{fig4} and \eqref{fig5}.
As we have done in the proof of Lemma \ref{hilbert} (see formulas from \eqref{chams00} to \eqref{chams5} above), we derive from the estimates \eqref{fig3} and \eqref{fig6} that
$\displaystyle \lim_{\varepsilon\rightarrow0}\|k_\varepsilon\|_\omega=0,$ which is the desired result of our lemma.
\end{proof}

The following lemma will be used in the proof of both Proposition \ref{lem22} and Proposition \ref{lem2}.

\begin{lem}\label{fal}
Let $\cI\subseteq\cL$ be  a closed ideal and let $g\in\cJ_{\omega}(\E_{_\cI})$ be a function, where $\omega$ is a modulus of continuity satisfying \eqref{cond}.
Let $\Gamma\in\Omega_{_{\E_{_\cI}}}$  be such that $\T\setminus\{\Gamma\cup\E_{_\cI}\}$ is a finite union of arcs $(a,b)\subseteq\T\setminus \E_{_\cI}$ $(a,b\in \E_{_\cI})$. We suppose that
 $gf\in\cI$ for some outer function $f\in\cL.$ Then  $gf_{\Gamma}\in\cI.$
\end{lem}

\begin{proof} For simplicity we suppose that $\T\setminus\{\Gamma\cup\E_{_\cI}\}=(a,b)=:\gamma,$ where
$a,b\in \E_{_\cI}$ and $(a,b)\subseteq\T\setminus \E_{_\cI}.$
Let $\varepsilon$ be a nonzero real number such that $\gamma_{\varepsilon}:=(a e^{i\varepsilon},b e^{-i\varepsilon})\subset\gamma.$
According to Lemma \ref{fliyo}, the functions
$\phi_{{\rho,\varepsilon}}f_{\gamma_{\varepsilon}}$
and $\phi_{{\rho,\varepsilon}}
f_{{\T\setminus\gamma_{\varepsilon}}}$ belong to $\cL,$ since $\phi_{{\rho,\varepsilon}}\in\cI_{\omega}(\{a e^{i\varepsilon}, b e^{-i\varepsilon}\}).$
Now, for $\pi:\cL\to \cL/\cI$
being the canonical quotient map, we clearly have
\begin{eqnarray*} 0=\pi\big(\phi^{2}_{{\rho,\varepsilon}}gf\big)
=\pi\big(\phi_{{\rho,\varepsilon}}gf_{{\T\setminus\gamma_{\varepsilon}}}\big)\pi
\big(\phi_{{\rho,\varepsilon}}f_{{\gamma_{\varepsilon}}}\big).
\end{eqnarray*}
Then $\pi(\phi_{{\rho,\varepsilon}}gf_{{\T\setminus\gamma_{\varepsilon}}})=0,$
since the function
$\pi\big(\phi_{{\rho,\varepsilon}}f_{{\gamma_{\varepsilon}}}\big)$ is
invertible. It follows $\phi_{{\rho,\varepsilon}}gf_{{\T\setminus\gamma_{\varepsilon}}}\in\cI.$
Since $g\in\cJ_{\omega}(\E_{_\cI})$ then $gf_{\Gamma}\in\cJ_{\omega}(\E_{_\cI}),$ by applying again Lemma \ref{fliyo}.
Using Lemma \ref{hamam} and the following inequality
\begin{equation*}
\|\phi^{2}_{{\rho,0}}\phi^{}_{{\rho,\varepsilon}}
gf_{{\T\setminus\gamma_{\varepsilon}}}-\phi^{3}_{{\rho,0}}gf_{{\Gamma}}\|_{\omega}\leq \|gf_{\Gamma}\|_\omega\times
\|\phi^{2}_{{\rho,0}}\phi^{}_{{\rho,\varepsilon}}
f_{{\gamma\setminus\gamma_{\varepsilon}}}-\phi^{3}_{{\rho,0}}\|_{\omega},
\end{equation*}
we obtain
\begin{equation}\label{hamam2}
\lim\limits_{\varepsilon\rightarrow0}\|\phi^{2}_{{\rho,0}}\phi^{}_{{\rho,\varepsilon}}
gf_{{\T\setminus\gamma_{\varepsilon}}}-\phi^{3}_{{\rho,0}}gf_{{\Gamma}}\|_{\omega}=0.
\end{equation}
We deduce that $\phi^{3}_{{\rho,0}}gf_{{\Gamma}}\in\cI,$ for all $\rho>0.$
On the other hand
$$\lim\limits_{\rho\rightarrow0}\|\phi^{3}_{{\rho,0}}gf_{{\Gamma}}-gf_{{\Gamma}}\|_{\omega}=0,$$
by applying Lemma \ref{lem3}. Therefore  $gf_{{\Gamma}}\in\cI.$
\end{proof}

\subsection{Proof of Proposition \ref{lem22}\label{sect42}}

Let $\cI\subseteq\cL$ be a closed ideal and let $g\in\cJ_\omega(\E_{_{\cI}})$ be
a function. We suppose that $gf\in\cI,$  for some outer function $f\in\cL.$
We set $\T\setminus\E_{_{\cI}}=\bigcup\limits_{n\in\N}(a_n,b_n),$ where $(a_n,b_n)\subseteq\T\setminus\E_{_{\cI}}$ and $a_n,b_n\in\E_{_{\cI}}.$
We have $\lim\limits_{n\rightarrow +\infty}\|g f_{{\Gamma_{n}}}-g\|_{\omega}=0,$ by applying Lemma \ref{fliyo} with $\Gamma_{n}=\bigcup\limits_{m> n}(a_m,b_m).$  Then $g\in\cI,$ since, by applying Lemma \ref{fal}, $gf_{\Gamma_n}\in\cI$ for all $n\in\N.$  $\hfill\qed$

\section{\bf Proof of Proposition \ref{lem2}. \label{sect3}}

To prove the proposition we need first to establish some lemmas.

\begin{lem}\label{multi}
Let $f\in\cL$ be a function, where $\omega$ is a modulus of continuity satisfying \eqref{cond}.  Let $\E$ be a closed subset of $\E_{f}$ and let $S\in\cH^{\infty}$ be a singular inner function such that $\sigma(S)\subseteq \E$ and $\int_{\T}d\mu_{_S}(e^{i\theta})\leq M,$ where $M>0$ is a constant. Then
 $Sf\in\cL$ and
 \begin{equation}\label{singular}
 \|Sf\|_{\omega}\leq c_{\omega,M} \|f\|_{\omega},
 \end{equation}
  where $c_{\omega,M}>0$ is a constant not depending on $f.$
  If moreover $f\in\cJ_\omega(\E)$ then $Sf\in\cJ_\omega(\E).$
 \end{lem}

\begin{proof}
We let $\xi,\zeta\in\T\setminus\E$ be two distinct points such that $d(\xi,\E)\leq d(\zeta,\E).$
 Since
 \begin{eqnarray}\label{chemis0}
 \frac{Sf(\xi)-Sf(\zeta)}{\omega(|\xi-\zeta|)}
= f(\xi)\frac{S(\xi)-S(\zeta)}{\omega(|\xi-\zeta|)}+S(\zeta)\frac{f(\xi)-f(\zeta)}{\omega(|\xi-\zeta|)}
 \end{eqnarray}
 then
  \begin{eqnarray}\label{chemis1}
 \Big|\frac{Sf(\xi)-Sf(\zeta)}{\omega(|\xi-\zeta|)}\Big|
 \leq \Big|f(\xi)\frac{S(\xi)-S(\zeta)}{\omega(|\xi-\zeta|)}\Big|+\|f\|_{\omega}.
  \end{eqnarray}
 We have two situations:

\textit{Case A.}  We suppose that $\displaystyle|\xi-\zeta|\geq \frac{1}{4} d^{2}(\xi,\E).$ Since $\omega$ satisfies \eqref{cond},
  \begin{eqnarray}\nonumber
  \Big|f(\xi)\frac{S(\xi)-S(\zeta)}{\omega(|\xi-\zeta|)}\Big|&\leq& 2\frac{|f(\xi)|}{\omega(\frac{1}{4}d^{2}(\xi,\E))}
   \\\label{chemis20}&\leq&  4\eta_{\omega}^{-1}\frac{|f(\xi)|}{\omega(d(\xi,\E))}
   \\\label{chemis2}&\leq&  4\eta_{\omega}^{-1}\|f\|_{\omega}.
  \end{eqnarray}

\textit{Case B.}  We now suppose that $\displaystyle|\xi-\zeta|\leq \frac{1}{4}d^{2}(\xi,\E).$ In this case we let $(\xi,\zeta)$ be the open arc joining the points $\xi$ and $\zeta,$ and such that $(\xi,\zeta)\subseteq\T\setminus\E.$ Since we have assumed $d(\xi,\E)\leq d(\zeta,\E),$ then
$$d(z,\E)\geq d(\xi,\E),\qquad z\in(\xi,\zeta).$$
We have
  \begin{eqnarray}
\nonumber
\frac{|S(\xi)-S(\zeta)|}{|\xi-\zeta|}\lesssim\sup_{z\in(\xi,\zeta)}|S'(z)|\lesssim\sup_{z\in(\xi,\zeta)}a_{_S}(z),
  \end{eqnarray}
where
$$a_{_S}(z):=\frac{1}{\pi}\int_{\T}\frac{1}{|e^{i\theta}-z|^2}d\mu_{_S}(e^{i\theta}),\qquad z\in\T\setminus\E.$$
Then
  \begin{eqnarray*}
\frac{|S(\xi)-S(\zeta)|}{|\xi-\zeta|}\lesssim  \frac{M}{d^{2}(\xi,\E)},
  \end{eqnarray*}
and consequently
  \begin{eqnarray}\nonumber
  \Big|f(\xi)\frac{S(\xi)-S(\zeta)}{\omega(|\xi-\zeta|)}\Big|&\lesssim& M \frac{|f(\xi)|}{d^{2}(\xi,\E)}\frac{|\xi-\zeta|}{\omega(|\xi-\zeta|)}.
  \end{eqnarray}
Therefore, by using the facts that $t/\omega(t)$ is nondecreasing and that $\omega$ satisfies \eqref{cond},
  \begin{eqnarray}\label{chemis30}
  \Big|f(\xi)\frac{S(\xi)-S(\zeta)}{\omega(|\xi-\zeta|)}\Big|&\lesssim& M \eta_{\omega}^{-1} \frac{|f(\xi)|}{\omega(d(\xi,\E))}
  \\\label{chemis3}&\lesssim&  M \eta_{\omega}^{-1}\|f\|_{\omega}.
  \end{eqnarray}
Using the inequalities \eqref{chemis1}, \eqref{chemis2} and \eqref{chemis3} we deduce that $\|Sf\|_{\omega}\leq c_{\omega,M} \|f\|_{\omega},$ and hence $Sf\in\cL.$ It remains to prove that if $f\in\cJ_{\omega}(\E)$ then so is for $Sf.$
Using \eqref{chemis0}
  \begin{eqnarray}\nonumber
&& \sup_{\xi,\zeta\in\E(\delta)\atop \xi\neq\zeta}\Big|\frac{Sf(\xi)-Sf(\zeta)}{\omega(|\xi-\zeta|)}\Big|
 \\\label{chemis4}&\leq&
 \sup_{\xi,\zeta\in\E(\delta)\setminus\E\atop \xi\neq\zeta}\Big|f(\xi)\frac{S(\xi)-S(\zeta)}{\omega(|\xi-\zeta|)}\Big|
 + \sup_{\xi,\zeta\in\E(\delta)\atop \xi\neq\zeta}\Big|\frac{f(\xi)-f(\zeta)}{\omega(|\xi-\zeta|)}\Big|,
  \end{eqnarray}
  for every $\delta\in]0,1].$ The desired result follows by combining \eqref{chemis20}, \eqref{chemis30} and \eqref{chemis4} and using the hypothesis that $f\in\cJ_{\omega}(\E).$
\end{proof}

For a function $f\in\cH^{\infty},$ we denote by $B_{f}$ and $S_{f}$ respectively  the Blaschke product associated to $\Z_{U_{_f}},$ and the singular part of $U_f.$ Similarly, $B_{_\cI}$ and $S_{_\cI}$ are respectively  the Blaschke product constructed from $\Z_{U_{_\cI}},$ and the singular part of $U_{_\cI}.$ For completeness, we give the proof of the  following classical lemma.

\begin{lem}\label{multi2}
Let $\cI\subseteq\cL$ be a closed ideal, where $\omega$ is an arbitrary modulus of continuity. There exists a sequence of functions
$\{g_k\ : \ k\in\N\}\subseteq\cI$ satisfying the following properties:
\begin{itemize}
  \item [$i.$] The greater common divisor of the singular functions $S_{g_{k}},$ $k\in\N,$  is equal to $S_{_\cI}$ and
  $$\lim_{k\rightarrow \infty}\|\mu_{_{S_{g_{k}}}}-\mu_{_{S_{_\cI}}}\|=0.$$
  \item [$ii.$]The sequence $\{B_{g_{k}}\ :\ k\in\N\}$ converges uniformly to $B_{_\cI}$  on each compact subset of $\D.$
\end{itemize}
\end{lem}

\begin{proof}
We denote by $\overline{\cI}$ the closedness of $\cI$ in $\cA(\D).$ According to Beurling-Rudin Theorem, we have
\begin{equation}\label{br}
\overline{\cI}:=\{g\in\cA(\D)\ :\ g/U_{_\cI}\in\cH^{\infty} \text{ and }g(z)=0,\ \text{for every } z\in\E_{_\cI}\}.
\end{equation}
Let now $O\in\cA(\D)$ be an outer function vanishing on $\E_{_\cI}.$ Since $\E_{_\cI}\supseteq\sigma(U_{_\cI})\cap\T,$ then $U_{_\cI}O\in\cA(\D).$
Thus $U_{_\cI}O\in\overline{\cI},$ by using \eqref{br}. We suppose that $U_{_\cI}O\notin\cI,$ otherwise the lemma is obvious. By applying again \eqref{br}, there exists a sequence of functions
$\{f_k\ : \ k\in\N\}\subseteq\cI$  converging uniformly to $U_{_\cI}O.$ As it is done in  \cite[Lemma 1.3]{ESZ}, we can extract a subsequence
$\{g_k\ : \ k\in\N\}$ of $\{f_k\ : \ k\in\N\}$ satisfying the desired results of the lemma.
\end{proof}

We also give the proof of the following simple lemma.

\begin{lem}\label{multi3}
Let $\{S_k:\ k\in\N\}$ be a sequence of singular functions such that $\text{supp}(\mu_{S_k})\subseteq \F,$ for every $k\in\N,$ where $\F$ is a closed subset of $\T.$
We suppose that $$\lim_{k\rightarrow \infty}\|\mu_{_{S_k}}-\mu_{_S}\|=0,$$
for some singular function $S.$ Then $\{S_k: \ k\in\N\}$ converges uniformly to $S$ on each compact subset of $\overline{\D}\setminus\F.$
\end{lem}

\begin{proof}
We let $\K$ be a compact subset of $\overline{\D}\setminus\F.$
It is obvious that
\begin{eqnarray}\label{moin}
\Big|\int_{\T}\frac{e^{i\theta}+z}{e^{i\theta}-z}d\mu_{_{S_{k}}}(e^{i\theta})-
\int_{\T}\frac{e^{i\theta}+z}{e^{i\theta}-z}d\mu_{_{S}}(e^{i\theta})\Big|
\leq\frac{2\|\mu_{_{S_k}}-\mu_{_S}\|}{\delta_{\K}},\qquad z\in\K,
\end{eqnarray}
where $\delta_{\K}:=\inf\limits_{z\in\K}d(z,\F).$ We remark that $\delta_{\K}>0.$
Using the estimate \eqref{moin} and the simple fact that there exists a neighborhood $\mathcal{U}$ of $0$ such that
$$|e^{z}-1|\leq 2|z|,\qquad z\in\mathcal{U},$$
we deduce that for sufficiently large numbers $k\in\N,$
\begin{eqnarray}
|S_k(z)-S(z)|\leq \Big|\exp\Big\{\frac{1}{2\pi}\int_{\T}\frac{e^{i\theta}+z}{e^{i\theta}-z}d(\mu_{_S}-\mu_{_{S_k}})(e^{i\theta})\Big\}-1\Big| \leq \frac{4\|\mu_{_{S_k}}-\mu_{_S}\|}{\delta_{\K}},
\end{eqnarray}
for all points $z\in\K.$ Which gives the desired result of the lemma.
\end{proof}

With a singular function $S\in\cH^{\infty}$ and a closed subset $\K\subseteq \T$ we associate the following singular function $$\big(S\big)_{\K}(z):=\exp\Big\{-\frac{1}{2\pi}\int_{\K}\frac{e^{i\theta}+z}{e^{i\theta}-z}d\mu_{_S}(e^{i\theta})\Big\},\qquad z\in\D.$$

\subsection*{ \it Proof of Proposition \ref{lem2}\label{sect32}}

Let $f\in\cI$ be a function. We define $B_{\{f,n\}}$ and $B_{\{\cI,n\}}$ to be respectively the Blaschke product with zeros $\Z_{U_f}\cap\D_{n}$ and  $\Z_{U_{_\cI}}\cap\D_{n},$ where
$$\D_{n}:=\{z\in\D\text{ : }|z|<\frac{n}{n+1}\}, \qquad n\in\N\setminus\{0\}.$$
For a fixed $n\in\N\setminus\{0\},$ we set
$$\cJ_n:=\{g\in\cL\ :\ gB_{\{\cI,n\}}\in\cI\}.$$
Let $\{g_k\ : \ k\in\N\}\subseteq\cI$ be the sequence of Lemma \ref{multi2}.
Since $\Z_{B_{\{\cI,n\}}}$ is finite,
then $B_{\{g_{k},n\}}=B_{\{\cI,n\}}$ for a sufficiently large number $k\in\N,$
by using the property $(ii)$ of Lemma \ref{multi2}.
Then $g_k/B_{\{g_{k},n\}}\in\cJ_n,$ for a large $k\in\N.$
Thus the zero set $\bigcap\limits_{g\in\cJ_n} g^{-1}(0)$ of the ideal $\cJ_n$ is contained in $\overline{\D}\setminus\D_{n},$ since
$g_k/B_{\{g_{k},n\}}$ does not vanish on $\D_{n}.$
Therefore $B_{\{f,n\}}$ is invertible in the quotient algebra $\cL/\cJ_n.$

We now let $\pi_n:\cL\to \cL/\cJ_n$
be the canonical quotient map. Since $f\in\cI$ then $f\in\cJ_n,$ and hence
\begin{eqnarray*} 0=\pi_n\big(f\big)
=\pi_n\big(f/B_{\{f,n\}}\big)\pi_n \big(B_{\{f,n\}}\big).
\end{eqnarray*}
Therefore $\pi_n\big(f/B_{\{f,n\}}\big)=0,$ and consequently
$f/B_{\{f,n\}}\in\cJ_n.$ It follows that
$$B_{\{\cI,n\}}(f/B_{\{f,n\}})\in\cI.$$
Since the sequence $\{B_{\{\cI,n\}}(B_f/B_{\{f,n\}})\ :\ n\in\N\}$ converges uniformly on compact subsets of $\overline{\D}\setminus\E_f$ to $B_{_\cI},$
and since $S_{f}O^2_{f}$ is continuous on $\overline{\D}$ and vanishes on $\E_f,$ then
$$\lim\limits_{n\rightarrow\infty}\|B_{\{\cI,n\}}(B_f/B_{\{f,n\}})S_{f}O^2_{f}-B_{_\cI}S_{f}O^2_{f}\|_{\infty}=0.$$
Thus
$$\lim\limits_{n\rightarrow\infty}\|B_{\{\cI,n\}}(B_f/B_{\{f,n\}})S_{f}O^2_{f}-B_{_\cI}S_{f}O^2_{f}\|_{\omega}=0,$$
by using Theorem \ref{main2} and the fact that $fO_f\in\cJ_\omega(\E_f),$ as product of two functions $f,O_f\in\cI_\omega(\E_f).$
Therefore
$$B_{_\cI}S_{f}O^2_{f}\in\cI.$$

We let $\varepsilon$ be a nonzero real number such that $\gamma_{\varepsilon}:=(a e^{i\varepsilon},b e^{-i\varepsilon})\subset\gamma:=(a,b),$ where
$(a,b)\subseteq\T\setminus \E_{_\cI}$ is the open arc joining the points $a,b\in \E_{_\cI}.$
Using Lemma \ref{fliyo}, we have $L_{\varepsilon}:=p_{{\varepsilon}}\times \big(O_f\big)_{\gamma_{\varepsilon}}\in\cL,$  where
$$\displaystyle p_{{\varepsilon}}(z):=(z\overline{a}e^{-i\varepsilon}-1)
(z\overline{b}e^{i\varepsilon}-1),\qquad z\in\overline{\D}.$$
It is obvious that $$\E_{{L_{\varepsilon}}}=\{a e^{i\varepsilon},b e^{-i\varepsilon}\}\cup\big(\overline{\gamma_{\varepsilon}}\cap\E_f\big).$$
Since $\sigma((S_{f})_{{\gamma_{\varepsilon}}})\subseteq \overline{\gamma_{\varepsilon}}\cap\E_f\subseteq\E_{{L_{\varepsilon}}},$ then $\big(S_{f}\big)_{\gamma_{\varepsilon}}L_{\varepsilon}\in\cL,$ by applying Lemma \ref{multi}. Furthermore, we remark that $B_{_\cI}\big(S_{f}\big)_{\T\setminus\gamma_{\varepsilon}}O_{f}\in\cL,$
by using the F-property of $\cL.$
We just proved above that $B_{_\cI}S_{f}O^2_{f}\in\cI,$ then
\begin{eqnarray*}
0=\pi\big(L_{\varepsilon}B_{_\cI}S_{f}O^2_{f}\big)=
\pi\big(\big(S_{f}\big)_{\gamma_{\varepsilon}}L_{\varepsilon}\big)
\times \pi\big(B_{_\cI}\big(S_{f}\big)_{\T\setminus\gamma_{\varepsilon}}O^2_{f}\big),
\end{eqnarray*}
where $\pi:\cL\to \cL/\cI$ is the canonical quotient map.
The function $\big(S_{f}\big)_{\gamma_{\varepsilon}}L_{\varepsilon}$ is invertible in the
quotient algebra $\cL/\cI,$ since its zero set $\E_{{L_{\varepsilon}}}$ does not intersect with the spectrum $\sigma(U_{_\cI})\cup \E_{_\cI}$ of the Banach algebra $\cL/\cI.$ Thus $\pi\big(B_{_\cI}\big(S_{f}\big)_{\T\setminus\gamma_{\varepsilon}}O^2_{f}\big)=0,$ and hence
$$B_{_\cI}\big(S_{f}\big)_{\T\setminus\gamma_{\varepsilon}}O^2_{f}\in\cI.$$

Using Lemma \ref{multi3}, we deduce that the sequence of elements $\big(S_{f}\big)_{\T\setminus\gamma_{\varepsilon}}$ converges uniformly on compact subsets of $\overline{\D}\setminus\E_f$ to $\big(S_{f}\big)_{\T\setminus\gamma},$ when $\varepsilon$ goes to $0.$
This fact and the clearly fact that $B_{_\cI}O^2_{f}$ is continuous on $\overline{\D}$ and vanishes on $\E_f,$ give simply
$$\lim\limits_{\varepsilon\rightarrow 0}\|B_{_\cI}\big(S_{f}\big)_{\T\setminus\gamma_{\varepsilon}}O^2_{f}-B_{_\cI}\big(S_{f}\big)_{\T\setminus\gamma}O^2_{f}\|_{\infty}=0.$$
Then
$$\lim\limits_{\varepsilon\rightarrow0}
\|B_{_\cI}\big(S_{f}\big)_{\T\setminus\gamma_{\varepsilon}}O^2_{f}-B_{_\cI}\big(S_{f}\big)_{\T\setminus\gamma}O^2_{f}\|_{\omega}=0,$$
by using Theorem \ref{main2}.
Thus $$B_{_\cI}\big(S_{_f}\big)_{\T\setminus\gamma}O^2_{f}\in\cI.$$
We now argue similarly  to arrive at
$$B_{_\cI}\big(S_{f}\big)_{\T\setminus\Delta_{_N}}O^2_{f}\in\cI,$$
where $\Delta_{N}:=\bigcup\limits_{n\leq N}(a_n,b_n)\in\Omega_{_{\E_{_\cI}}}.$
As it is just done above, a simple application of Lemma \ref{multi3} gives
$$\lim\limits_{N\rightarrow \infty}\|B_{_\cI}\big(S_{f}\big)_{\T\setminus\Delta_{_N}}O^2_{f}-B_{_\cI}\big(S_{f}\big)_{\E_{_\cI}}O^2_{f}\|_{\infty}=0.$$
Then
$$\lim\limits_{N\rightarrow \infty}\|B_{_\cI}\big(S_{f}\big)_{\T\setminus\Delta_{_N}}O^2_{f}-B_{_\cI}\big(S_{f}\big)_{\E_{_\cI}}O^2_{f}\|_{\omega}=0,$$
by using again Theorem  \ref{main2}.
Therefore $$B_{_\cI}\big(S_{f}\big)_{\E_{_\cI}}O^2_{f}\in\cI.$$
We deduce
\begin{equation*}
B_{_\cI}\big(S_{g_k}\big)_{\E_{_\cI}}O^2_{g_k}\in\cI,\qquad \text{for every } k\in\N,
\end{equation*}
and hence
\begin{equation} \label{ded0}
B_{_\cI}\big(S_{g_k}\big)_{\E_{_\cI}}O^2_{g_k}O^2_{f}\in\cI,\qquad \text{for every } k\in\N,
\end{equation}
where $\{g_k\ : \ k\in\N\}\subseteq\cI$ is the sequence of  Lemma \ref{multi2}.
Using respectively the F-Property of $\cL$ and Lemma \ref{multi}, we obtain $B_{_\cI}O_{f}\in\cL$ and
$\big(S_{g_k}\big)_{\E_{_\cI}}O_{f}\in\cL.$ Since this two functions vanish on $\E_{_\cI},$ we obtain
$B_{_\cI}\big(S_{g_k}\big)_{\E_{_\cI}}O^2_{f}\in\cJ_\omega(\E_{_\cI}).$
Thus, by applying Proposition \ref{lem22},
$$B_{_\cI}\big(S_{g_k}\big)_{\E_{_\cI}}O^2_{f}\in\cI,\qquad \text{for every } k\in\N,$$
since we proved \eqref{ded0}. We deduce that the singular functions $\big(S_{g_k}\big)_{\E_{_\cI}}$ are all divided by $S_{_\cI},$
and since
$$\|\mu_{{\big(S_{g_{k}}\big)_{\E_{_\cI}}}}-\mu_{_{S_{_\cI}}}\|\leq\|\mu_{_{S_{g_{k}}}}-\mu_{_{S_{_\cI}}}\|,\qquad \text{for every } k\in\N,$$
then, by using Lemma \ref{multi2},
$$\lim_{k\rightarrow \infty}\|\mu_{{\big(S_{g_{k}}\big)_{\E_{_\cI}}}}-\mu_{_{S_{_\cI}}}\|=0.$$
Thus
$$\lim\limits_{k\rightarrow \infty}\|B_{_\cI}\big(S_{g_k}\big)_{\E_{_\cI}}O^2_{f}-U_{_\cI}O^2_{f}\|_{\infty}=0,$$
by using Lemma \ref{multi3} and also the simple fact that $B_{_\cI}O^2_{f}$ is continuous on $\overline{\D}$ and vanishes on $\E_{_\cI}.$
Then
$$\lim\limits_{k\rightarrow \infty}\|B_{_\cI}\big(S_{g_k}\big)_{\E_{_\cI}}O^2_{f}-U_{_\cI}O^2_{f}\|_{\omega}=0,$$
by using once more Theorem  \ref{main2}. Hence $U_{_\cI}O^2_{f}\in\cI,$ which gives the desired result
of the proposition.

\section{\bf Proof of Theorem \ref{main2}. \label{sect5}}

Before giving the proof of the theorem, we need first to establish some needed results. Some of these results are actually inspired from
\cite{Bou1,Shi,HS}.

\subsection{Some properties of functions from  $\cJ_{\omega}(\E)$\label{sect51}}

The next lemma will simplify the proof of the following one.

\begin{lem}\label{form0}
Let $g\in \cJ_{\omega}(\E)$ be a function, where $\omega$ is an arbitrary modulus of continuity. For every $\nu\geq0,$ we have
\begin{eqnarray}\nonumber
&&\exp\Big\{
\frac{1}{2\pi}\int_{\xi\in\T}\frac{1-|z|^2}{|\xi-z|^2}\log\big(|g(\xi)-g(z/|z|)|+\nu\big)|d\xi|\Big\}
\\\label{matin2}&\leq& o(\omega(1-|z|))+A\nu,\qquad\text{as } d(z,\E)\rightarrow 0\text{ and } z\in\D,
\end{eqnarray}
where $A\geq 1$ is an absolute constant.
\end{lem}

\begin{proof}
Since $g\in \cJ_{\omega}(\E),$
\begin{equation}\label{matin0}
\lim_{\delta\rightarrow 0}\sup_{\xi,\zeta\in\E(\delta)\atop \xi\neq\zeta}\frac{|g(\xi)-g(\zeta)|}{\omega(|\xi-\zeta|)}=0.
\end{equation}
Let $\varepsilon\leq1$ be a positive number. By using \eqref{matin0}, there exists a positive number
$c_\varepsilon$  such that
$$|g(\xi)-g(z/|z|)|\leq \varepsilon\ \omega(|\xi-z/|z||)$$
for every $z\in\D$ for which $d(z,\E)\leq c_\varepsilon$ and for every $\xi\in\T$ satisfying $|\xi-z/|z||\leq c_\varepsilon.$
For a point $z\in\D,$ we divide $\T$ into the following three parts
\begin{eqnarray*}
\Gamma_1(z)&:=& \big\{\xi\in\T \ :\ |\xi-z/|z||\leq 1-|z|\leq c_{\varepsilon}\big\},\\
\Gamma_2(z)&:=& \big\{\xi\in\T \ :\ 1-|z|\leq|\xi-z/|z||\leq c_{\varepsilon}\big\},\\
\Gamma_3(z)&:=& \big\{\xi\in\T \ :\ c_{\varepsilon}\leq|\xi-z/|z||\big\}.
\end{eqnarray*}
As in the proof of \cite[Lemma A.2]{Bou1} (with  $\delta=\nu$ ) we deduce that, for all points $z\in\D$ sufficiently close to $\E,$
\begin{eqnarray}
\nonumber&&\frac{1}{2\pi}\int_{\xi\in\T}\frac{1-|z|^2}
{|\xi-z|^2}\log\big(|g(\xi)-g(z/|z|)|+\nu\big)|d\xi|
\\\nonumber&\leq& \frac{1}{2\pi}\int_{\Gamma_1(z)}+\frac{1}{2\pi}\int_{\Gamma_2(z)}+\frac{1}{2\pi}\int_{\Gamma_3(z)}
\\\label{matin3}&\leq& \log\big(\varepsilon\ \omega(1-|z|)+\nu\big)+c \int_{t\geq 1}\frac{\log(t)}{t^2}dt
-\varepsilon\log(\varepsilon),
\end{eqnarray}
where $c>0$ is a constant. Hence \eqref{matin2} is deduced  from  \eqref{matin3}.
\end{proof}

An application of Lemma \ref{form0} provide us with the following proposition.

\begin{prop}\label{coincide}
Let $\omega$ be an arbitrary modulus of continuity. Then
$\cJ_{\omega}(\E)=\cK_{\omega}(\E),$
where
$$\cK_{\omega}(\E):=\Big\{f\in\cI_{\omega}(\E)\ :\
\lim_{\delta\rightarrow 0}\sup_{d(z,\E),d(w,\E)\leq\delta\atop z,w\in\D \text{ and } z\neq w}\frac{|f(z)-f(w)|}{\omega(|z-w|)}=0\Big\}.$$
\end{prop}

\begin{proof} It is obvious that $\cJ_{\omega}(\E)\supseteq\cK_{\omega}(\E).$
We now let $g\in\cJ_{\omega}(\E)$ be a function and $\varepsilon$ be a fixed positive number.
We remark that,  for $z,w\in\D$ such that $\min\{|z|,|w|\}\geq\frac{1}{2},$
$$|z-w|\geq \frac{1}{4}\Big|\frac{z}{|z|}-\frac{w}{|w|}\Big|.$$
For two distinct points $z,w\in\D,$  we have two possible cases:
$$\text{Case 1. } |z-w|\geq \frac{1}{2}\max\{1-|z|,1-|w|\}\quad \text{or}\quad \text{Case 2. } |z-w|\leq \frac{1}{2}\max\{1-|z|,1-|w|\}.$$
We let $z,w\in\D$ be two distinct points  such that $\min\{|z|,|w|\}\geq\frac{1}{2}.$
In the first case, by using the facts that both $\omega$ and $t\mapsto t/\omega(t)$ are nondecreasing functions, and \eqref{omega},
\begin{eqnarray}\label{final1}
\frac{|g(z)-g(w)|}{\omega(|z-w|)}\leq 2\frac{\big|g((z)-g(\frac{z}{|z|})\big|}{\omega(1-|z|)}
+ 4\frac{\big|g(\frac{z}{|z|})-g(\frac{w}{|w|})\big|}{\omega(|\frac{z}{|z|}-\frac{w}{|w|}|)}
+ 2\frac{\big|g(w)-g(\frac{w}{|w|})\big|}{\omega(|1-|w||)}.
\end{eqnarray}
Furthermore, by applying Lemma \ref{form0} with $\nu=0,$
\begin{eqnarray}\nonumber
\big|g(z)-g(\frac{z}{|z|})\big| &\leq& \exp\Big\{
\frac{1}{2\pi}\int_{\xi\in\T}\frac{1-|z|^2}{|\xi-z|^2}\log\big(\big|g(\xi)-g(\frac{z}{|z|})\big|\big)|d\xi|\Big\}
\\\label{final2}&=& o(\omega(1-|z|)),\qquad\text{as } d(z,\E)\rightarrow 0.
\end{eqnarray}
Thus, by combining the estimates \eqref{final1} and \eqref{final2}, and using our assumption that $g\in\cJ_{\omega}(\E),$ we obtain
\begin{eqnarray}\label{final3}
\frac{|g(z)-g(w)|}{\omega(|z-w|)}\leq \varepsilon,
\end{eqnarray}
if  $z,w\in\D$ are two distinct points sufficiently close to $\E,$ and satisfying $|z-w|\geq \frac{1}{2}\max\{1-|z|,1-|w|\}.$

In the second case, it can be assumed without loss of generality that  $\max\{1-|z|,1-|w|\}=1-|z|.$
Since $t\mapsto t/\omega(t)$ is nondecreasing, we
clearly have
\begin{eqnarray}\nonumber
\frac{|g(z)-g(w)|}{\omega(|z-w|)}\leq \Big(\sup_{|q-z|\leq \frac{1}{2}(1-|z|)}|g'(q)|\Big)\frac{|z-w|}{\omega(|z-w|)}
\\\label{final4}\leq\Big(\sup_{|q-z|\leq \frac{1}{2}(1-|z|)}|g'(q)|\Big)\frac{1-|z|}{\omega(1-|z|)}.
\end{eqnarray}
On the other hand, for a point $q\in\D$ satisfying $|q-z|\leq \frac{1}{2}(1-|z|),$  the classical Cauchy formula gives
\begin{eqnarray}\nonumber\label{final45}
g'(q)=\frac{1}{2i\pi}\int_{|p-z|= \frac{3}{4}(1-|z|)}\frac{f(p)-f(z)}{(p-q)^2}dp.
\end{eqnarray}
It follows
\begin{eqnarray}\label{final5}
|g'(q)|\leq \frac{1}{2\pi}\int_{|p-z|= \frac{3}{4}(1-|z|)}\frac{|f(p)-f(z)|}{|p-q|^2}|dp|.
\end{eqnarray}
Then
\begin{eqnarray}\label{final6}
|g'(q)|\leq \frac{12}{1-|z|}\sup_{|p-z|= \frac{3}{4}(1-|z|)}|f(p)-f(z)|.
\end{eqnarray}
Thus, by using \eqref{final3} and since $\omega$ is nondecreasing,
\begin{eqnarray}\label{final7}
|g'(q)|\leq \varepsilon\frac{\omega(1-|z|)}{1-|z|},
\end{eqnarray}
if  $z$ is sufficiently close to $\E.$
Therefore, by combining \eqref{final4} and \eqref{final7},
\begin{eqnarray}\label{final8}
\frac{|g(z)-g(w)|}{\omega(|z-w|)}\leq \varepsilon,
\end{eqnarray}
if  $z,w\in\D$ are two distinct points sufficiently close to $\E,$ and satisfying $|z-w|\leq \frac{1}{2}\max\{1-|z|,1-|w|\}.$

We now arrive to the conclusion that  $g\in\cK_{\omega}(\E),$ by joining together the estimates \eqref{final3} and \eqref{final8}. Hence  $\cJ_{\omega}(\E)\subseteq\cK_{\omega}(\E),$ which completes the proof of the proposition.
\end{proof}

For an inner function $U\in\cH^\infty$ we set
$$a_{U}(\xi):=\sum\limits_{n}\frac{1-|z_n|^2}{|\xi-z_n|^2}
+\frac{1}{\pi}\int_{\T}\frac{1}{|e^{i\theta}-\xi|^2}d\mu_{_{U}}(\theta),\qquad \xi\in\T\setminus\E_f,$$
where $z_n,$ $n\in\N,$ are the zeros of $U,$  each $z_n$ is repeated according to its multiplicity, and $\mu_{_U}$ is the positive singular measure associated with the singular factor of $U,$ with the understanding that $a_{U}=0$ if $U$ is constant.
In \cite{Shi3} Shirokov proved
\begin{equation}\label{shi}
\sup_{\zeta\in\T\setminus\E_f}
\frac{|f(\zeta)|}{\omega(\min\{1,a^{-1}_{U_f}(\zeta)\})}<+\infty,
\end{equation}
for every function $f\in\cL,$ see also \cite[Lemma B.5]{Bou1}. However, functions from the space $\cJ_{\omega}(\E)$ admit more  precise control than \eqref{shi};

\begin{lem}\label{forml}
Let $g\in \cJ_{\omega}(\E)$ be a function, where $\omega$ is an arbitrary modulus of continuity.
Then$$\lim_{\delta\rightarrow0}\sup_{\zeta\in\E(\delta)\setminus\E_g}
\frac{|g(\zeta)|}{\omega(\min\{1,a^{-1}_{U_g}(\zeta)\})}=0.$$
\end{lem}

\begin{proof}
We have
\begin{eqnarray}
&&|O_{g}(z)|\nonumber
\\\nonumber&=&\exp\Big\{\frac{1}{2\pi}\int_{\xi\in\T}\frac{1-|z|^2}{|\xi-z|^2}\log|g(\xi)||d\xi|\Big\}
\\\label{matin11}&\leq&\exp\Big\{
\frac{1}{2\pi}\int_{\xi\in\T}\frac{1-|z|^2}{|\xi-z|^2}\log\big(|g(\xi)-g(z/|z|)|+|g(z/|z|)|\big)|d\xi|\Big\},\quad z\in\D.
\end{eqnarray}
Then, using Lemma \ref{form0} with $\nu=|g(z/|z|)|,$
\begin{eqnarray}\label{matin22}
|O_{g}(z)|
\leq o(\omega(1-|z|))+A|g(z/|z|)|,\qquad\text{as } d(z,\E)\rightarrow 0.
\end{eqnarray}
From Proposition \ref{coincide},
\begin{eqnarray}\label{matin42}
|g(z)-g(\zeta)|= o(\omega(|z-\zeta|)),\qquad\text{as } d(z,\E),d(\zeta,\E)\rightarrow 0.
\end{eqnarray}
We let $\varepsilon\leq1$ be a positive number. Using  \eqref{matin22} and \eqref{matin42}, there exists a positive number $\delta_\varepsilon$ such that
\begin{equation}\label{equation1}
|O_{g}(z)|\leq \varepsilon\omega(|1-|z|) + A|g(z/|z|)|,
\end{equation}
and
\begin{equation}\label{equation2}
|g(z)-g(\zeta)|\leq \varepsilon\omega(|z-\zeta|),
\end{equation}
whenever $1-|z|\leq\delta_\varepsilon$, $d(z/|z|,\E)\leq\delta_\varepsilon$
and $d(\zeta,\E)\leq\delta_\varepsilon.$

Let $\zeta\in\E(\delta_\varepsilon)\setminus\E_g$ be  a point satisfying $$8Aa^{-1}_{U_g}(\zeta)< 1.$$ Two cases are possible;

\textit{Case A. } We assume that $\displaystyle d(\zeta, \Z_{g})\leq 8Aa^{-1}_{U_g}(\zeta).$ Using \eqref{equation2}
 $$\displaystyle|g(\zeta)|\leq \varepsilon\omega(d(\zeta, \Z_{g})).$$
Then
 $$\displaystyle|g(\zeta)|\leq\varepsilon \omega(8Aa^{-1}_{U_g}(\zeta))\leq 8A\varepsilon\omega(a^{-1}_{U_g}(\zeta)).$$

\textit{Case B.} We now assume that $\displaystyle d(\zeta, \Z_{g})\geq 8Aa^{-1}_{U_g}(\zeta).$
 Then $1-\rho_{_\zeta}\leq d(\zeta, \Z_{g}),$ where $$\displaystyle \rho_{_\zeta}:=1-8Aa^{-1}_{U_g}(\zeta).$$
 We have
\begin{equation}\label{equation3}
|U_{g}(\rho_{_\zeta} \zeta)|\leq\exp\big\{-\frac{1-\rho_{_\zeta}}{8}a_{U_g}(\zeta)\big\} =e^{-A},
\end{equation}
 as it is already computed in \cite[Lemma B.4]{Bou1}. We observe that
 $$1-\rho_{_\zeta}\leq d(\zeta, \Z_{g})\leq d(\zeta, \E_{g})\leq d(\zeta, \E)\leq \delta_\varepsilon.$$
From \eqref{equation1} and \eqref{equation3}
\begin{equation}\label{equation4}
|g(\rho_{_\zeta} \zeta)|=|U_g(\rho_{_\zeta} \zeta)||O_g(\rho_{_\zeta} \zeta)|\leq Ae^{-A}\big(\varepsilon\ \omega(1-\rho_{_\zeta})+|g(\zeta)|\big).
\end{equation}
Therefore
\begin{eqnarray*} |g(\zeta)|&\leq& |g(\zeta)-g(\rho_{_\zeta} \zeta)|+|g(\rho_{_\zeta} \zeta)|\\&\leq&
\varepsilon\ \omega(1-\rho_{_\zeta})+ Ae^{-A}\big(\varepsilon\ \omega(1-\rho_{_\zeta})+|g(\zeta)|\big).
\end{eqnarray*}
Hence
\begin{eqnarray*}
|g(\zeta)|&\leq&\varepsilon(1-Ae^{-A})^{-1}(1+Ae^{-A})\omega(1-\rho_{_\zeta})
\\&\leq& 3\varepsilon\ \omega(1-\rho_{_\zeta})
\\&\leq& 24 A\varepsilon\ \omega(a^{-1}_{U_g}(\zeta)).
\end{eqnarray*}
The lemma follows by joining together the results of the above two cases.
\end{proof}

\subsection{Proof of Theorem \ref{main2}\label{sect52}}

Let $g\in\cL$ be a function and $V\in\cH^{\infty}(\D)$ be an inner function dividing $g.$
We have $O_{g}\in \cL$ and $VO_{g}\in \cL,$ by the F-property of $\cL.$ Now, we suppose additionally  that  $g\in \cJ_{\omega}(\E).$
We will first show that
\begin{equation}\label{souklarba3}
\lim\limits_{\delta\rightarrow 0}\sup\limits_{ \xi,\zeta\in\E(\delta)\setminus\E_g\atop \xi\neq\zeta} \Big|g(\zeta)\frac{V(\xi)-V(\zeta)}{\omega(|\xi-\zeta|)}\Big|=0,
\end{equation}
uniformly with respect to any inner function $V$ dividing $U_g.$
Let $\xi,\zeta\in\T\setminus\E_g$ be two distinct points. We have two possibilities;

\textit{Case A.} We suppose that $\displaystyle|\xi-\zeta|\geq \frac{1}{2}d(\zeta,\Z_g).$  Since $g\in \cJ_{\omega}(\E),$ then
\begin{equation}\label{souk3}
\frac{|g(\zeta)|}{\omega(d(\zeta,\Z_g))}=o(1),\qquad \mbox{as }d(\zeta, \E)\rightarrow 0.
\end{equation}
by using \eqref{matin42}.
It follows
\begin{equation}\label{souk4}
\Big|g(\zeta)\frac{V(\xi)-V(\zeta)}{\omega(|\xi-\zeta|)}\Big|\leq 2
\frac{|g(\zeta)|}{\omega(d(\zeta,\Z_g))}=o(1),\qquad \mbox{as }d(\zeta, \E)\rightarrow 0.
\end{equation}

\textit{Case B.} We now suppose that $\displaystyle|\xi-\zeta|\leq \frac{1}{2}d(\zeta,\Z_g).$  Let  $(\xi,\zeta)\subset
\T\setminus \E_{g}$ be the arc joining the points $\xi$ and $\zeta.$ We have
\begin{eqnarray}\nonumber
            \frac{1}{2}|\zeta-w|\leq|z-w|\leq \frac{3}{2}|\zeta-w|,                \qquad z\in(\xi,\zeta)\text{ and }w\in\Z_g.
\end{eqnarray}
Since
\begin{eqnarray}\nonumber
\frac{|V(\xi)-V(\zeta)|}{|\xi-\zeta|}&\lesssim& \sup_{z\in(\xi,\zeta)} |V'(z)|
\\\nonumber&\lesssim& \sup_{z\in(\xi,\zeta)}\Big\{\sum\limits_{n}\frac{1-|z_n|^2}{|z-z_n|^2}
+\frac{1}{\pi}\int_{\T}\frac{1}{|e^{i\theta}-z|^2}d\mu_{_{U_g}}(\theta) \Big\},
\end{eqnarray}
where $\{z_n\ :\ n\in\N\}$ is the zero set of $g$ in $\D$ and $\mu_{_{U_g}}$ is the positive singular measure associated with the singular factor of $U_g.$
Then
\begin{eqnarray}\label{souk5}
\frac{|V(\xi)-V(\zeta)|}{|\xi-\zeta|}\lesssim a_{U_g}(\zeta).
\end{eqnarray}

\textit{  Case B.1.} We assume that  $\displaystyle\displaystyle |\xi-\zeta|\leq a^{-1}_{U_g}(\zeta).$
Since $t\mapsto t/\omega(t)$ is nondecreasing, and by applying Lemma \ref{forml}, we have
\begin{equation}\label{souk55}
|g(\zeta)|a_{U_g}(\zeta)\frac{|\xi-\zeta|}{\omega(|\xi-\zeta|)}\leq \frac{|g(\zeta)|}{\omega(\min\{1,a^{-1}_{U_g}(\zeta)\})}=
o(1),\qquad \mbox{as } d(\zeta, \E)\rightarrow 0.
\end{equation}
Using \eqref{souk5} and \eqref{souk55},
\begin{eqnarray}\nonumber
\Big|g(\zeta)\frac{V(\xi)-V(\zeta)}{\omega(|\xi-\zeta|)}\Big|&=& |g(\zeta)|\frac{|V(\xi)-V(\zeta)|}{|\xi-\zeta|}\frac{|\xi-\zeta|}{\omega(|\xi-\zeta|)}
\\&=&\label{souk6} o(1),\qquad \mbox{as } d(\zeta, \E)\rightarrow 0.
\end{eqnarray}

\textit{  Case B.2.}  We assume that  $\displaystyle\displaystyle |\xi-\zeta|\geq a^{-1}_{U_g}(\zeta).$ Then, by  applying once more Lemma \ref{forml},
\begin{eqnarray}\nonumber
\Big|g(\zeta)\frac{V(\xi)-V(\zeta)}{\omega(|\xi-\zeta|)}\Big|&\leq&
2\frac{|g(\zeta)|}{\omega(\min\{1,a^{-1}_{U_g}(\zeta)\})}
\\&=&\label{souk7} o(1), \qquad \mbox{as }d(\zeta, \E)\rightarrow 0.
\end{eqnarray}
Hence, the equality \eqref{souklarba3} is deduced by combining \eqref{souk4}, \eqref{souk6} and \eqref{souk7}.
We now clearly have
\begin{equation}\label{gat1}
\frac{VO_g(\xi)-VO_g(\zeta)}{\omega(|\xi-\zeta|)}
=V(\xi)\frac{O_g(\xi)-O_g(\zeta)}{\omega(|\xi-\zeta|)}
+O_{g}(\zeta)\frac{V(\xi)-V(\zeta)}{\omega(|\xi-\zeta|)},
\end{equation}
for all distinct points   $\xi,\zeta\in\T\setminus\E_g.$
By considering the particular case $V=U_g$ in the equality \eqref{gat1}, we obtain
\begin{eqnarray}\label{gat2}
\Big|\frac{O_g(\xi)-O_g(\zeta)}{\omega(|\xi-\zeta|)}\Big|
\leq\Big|\frac{g(\xi)-g(\zeta)}{\omega(|\xi-\zeta|)}\Big|+\Big|g(\zeta)\frac{U_g(\xi)-U_g(\zeta)}{\omega(|\xi-\zeta|)}\Big|.
\end{eqnarray}
Thus
\begin{eqnarray}\nonumber
&&\Big|\frac{VO_g(\xi)-VO_g(\zeta)}{\omega(|\xi-\zeta|)}\Big|
\\\label{gat3}&\leq&\Big|\frac{g(\xi)-g(\zeta)}{\omega(|\xi-\zeta|)}\Big|+\Big|g(\zeta)\frac{U_g(\xi)-U_g(\zeta)}{\omega(|\xi-\zeta|)}\Big|
+\Big|g(\zeta)\frac{V(\xi)-V(\zeta)}{\omega(|\xi-\zeta|)}\Big|,
\end{eqnarray}
for all distinct points   $\xi,\zeta\in\T\setminus\E_g.$ Therefore $VO_g\in\cJ_{\omega}(\E),$
since by hypothesis $g\in\cJ_{\omega}(\E),$ and combining the estimates \eqref{souklarba3} and \eqref{gat3}.
Which finishes the proof of \eqref{fprop1}. In particular we deduce that $O_g\in\cJ_{\omega}(\E).$

We now let
$\{V_n\ :\ n\in\N\}\subset\cH^{\infty}$ be a sequence of inner functions dividing $g,$ satisfying \eqref{fprop1.5} and such that $\sigma(V_n)\cap\T\subseteq\E,$ for every $n\in\N.$
Then $\sigma(V)\cap\T\subseteq\E.$
The equality \eqref{souklarba3} gives the following one,
\begin{equation}\label{sam}
\lim_{\delta\rightarrow 0}\sup_{\xi,\zeta\in\E(\delta)\setminus\E\atop
\xi\neq\zeta}\Big|O_g(\zeta)\frac{h_n(\xi)-h_n(\zeta)}{\omega(|\xi-\zeta|)}\Big|=0,
\end{equation}
uniformly with respect to $n\in\N,$  where $h_n:=V_n-V.$ Tacking account of \eqref{sam} and the fact that $O_g\in\cJ_{\omega}(\E),$ the equality \eqref{fprop2} is deduced by applying Lemma \ref{hilbert} with $\Delta_n=\T\setminus\E,$ for every $n\in\N.$ This completes the proof of the theorem.

\end{document}